\documentclass[11pt,oneside]{article}
\usepackage[utf8]{inputenc}
\usepackage[english]{babel}
\usepackage{amsmath}
\usepackage{amsfonts}
\usepackage{amssymb}
\usepackage{amsthm}
\usepackage{tikz-cd} 
\usepackage[left=3cm,right=3cm,top=3cm,bottom=3cm]{geometry}

\usepackage[]{algorithm2e}

\usepackage[normalem]{ulem}
\useunder{\uline}{\ul}{}
\theoremstyle{plain}
\newtheorem{teo}{}[section]
\newtheorem{prop}[teo]{Proposition}

\newtheorem{rem}[teo]{Remark}
\newtheorem{lem}[teo]{Lemma}
\newtheorem{thm}[teo]{Theorem}
\newtheorem{df}[teo]{Definition}
\theoremstyle{definition}
\newtheorem{ex}[teo]{Example}

\newcommand\blfootnote[1]{%
  \begingroup
  \renewcommand\thefootnote{}\footnote{#1}%
  \addtocounter{footnote}{-1}%
  \endgroup
}

\title{Computational approximations of compact metric spaces}
\author{Pedro J. Chocano, Manuel A. Mor\'on and Francisco R. Ruiz del Portal}
\date{}

\begin{document}
\maketitle

\begin{abstract}
Given a compact metric space $X$, we associate to it an inverse sequence of finite $T_0$ topological spaces. The inverse limit of this inverse sequence contains a homeomorphic copy of $X$ that is a strong deformation retract. We provide a method to approximate the homology groups of $X$ and other algebraic invariants. Finally, we study computational aspects and the implementation of this method.
\end{abstract}

\section{Introduction and preliminaries}\label{sec:introduccion}
\blfootnote{2020  Mathematics  Subject  Classification: 06A06,  	06A11,  55N05,  	55Q07,   	54E45}
\blfootnote{Keywords: finite topological spaces, compact metric spaces, approximation, homotopy, homology groups.}
\blfootnote{This research is partially supported by Grants PGC2018-098321-B-100 and BES-2016-076669 from Ministerio de Ciencia, Innovación y Universidades (Spain).}

Approximation of topological spaces is an old theme in geometric topology and it can have applications to the study of dynamical systems. For example, the study of dynamical objects such as attractors or repellers. In general, these objects do not have a good local behavior and therefore it can be difficult a direct study of them. For this reason, it is important to develop a theory of approximation based on finite data that can be obtained from experiments. Let $X$ be a compact metric space. Then there are two classical approaches to approximate $X$ from a theoretical point of view. One approach is to find a simpler topological space $Y$ such that $X$ and $Y$ share some topological properties (compactness, homotopy type, etc.) or algebraic properties (homology and homotopy groups, etc.). Polyhedra have been used to this aim, see for instance \cite{mardevsic2005approximating}. We recall a classical result, which is known as the nerve theorem. The idea is to use good covers to construct a simplicial complex that reconstructs the homotopy type of $X$, see \cite{borsuk1948imbedding} and \cite{mccord1967homotopy} for more details. Given an open cover $\mathcal{U}$ for a compact metric space $X$, the nerve of $\mathcal{U}$ is a simplicial complex $N(\mathcal{U})$ such that its vertices are the elements $U$ of $\mathcal{U}$ and $U_0,...,U_n\in \mathcal{U}$ span a simplex of $N(\mathcal{U})$ whenever $U_0\cap ... \cap U_n\neq \emptyset$.
\begin{thm}[Nerve theorem]\label{thm:nerveLemma} If $\mathcal{U}$ is an open cover of a compact space $X$ such that every non-empty intersection of finitely many sets in $\mathcal{U}$ is contractible, then $X$ is homotopy equivalent to the nerve of $\mathcal{U}$.
\end{thm} 
\begin{ex} We consider the topological space $X\subseteq \mathbb{R}^2$ given in Figure \ref{figure:nervelemaejemplo}. Let $\mathcal{U}$ denote the open cover given by $U_1,U_2,U_3,U_4$ and all the possible intersections of them. It is clear that $N(\mathcal{U})$ has the same homotopy type of $X$, see Figure \ref{figure:nervelemaejemplo}.
\begin{figure}[h]
\centering
\includegraphics[scale=1]{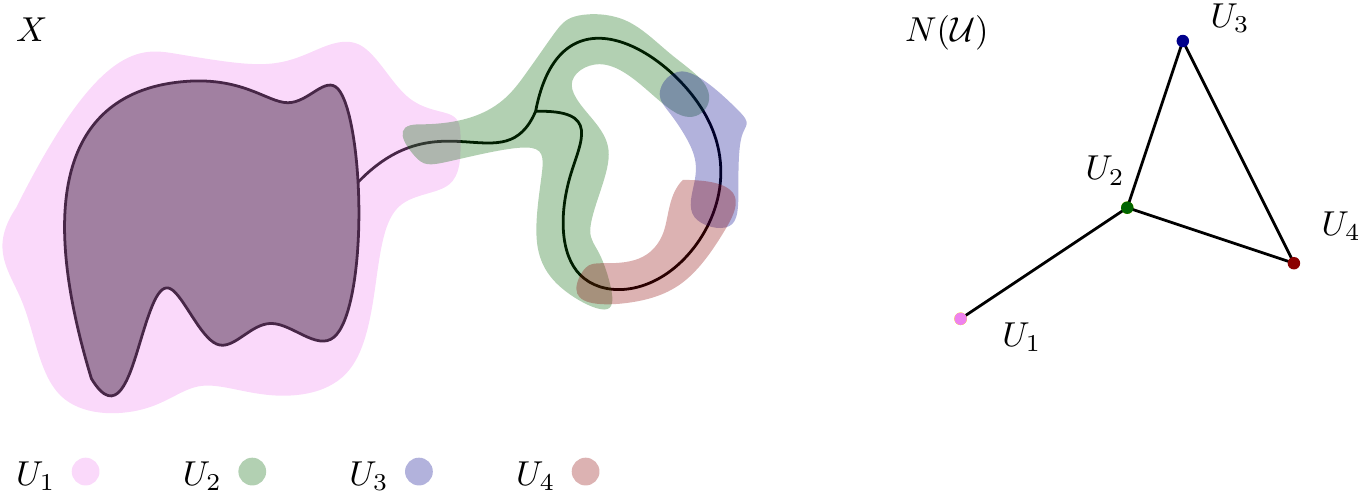}
\caption{$X$, schematic representation of $\mathcal{U}$ and $N(\mathcal{U})$.}\label{figure:nervelemaejemplo}
\end{figure}
\end{ex}
This approach has a drawback. It is not always easy to find an open cover satisfying the hypothesis of the nerve theorem. Recently, interesting results have been obtained modifying the notion of good cover. In \cite{govc2017approximate}, the hypothesis of being a good cover is relaxed. Then using persistent homology, results about the reconstruction of the homology of the original space are obtained. Using Vietoris-Rips complexes,  results of reconstruction have been obtained for Riemannian manifolds in \cite{hausmann1995onVietoris}.

A different approach is to approximate $X$ studying the inverse limit of an inverse sequence. We briefly recall the notion of inverse sequence and inverse limit, for a complete exposition see \cite{mardevsic1982shape}. An inverse sequence of finite $T_0$ topological spaces $(X_n,q_{n,n+1})$ consists of a sequence of finite $T_0$ topological spaces $\{ X_n\}_{n\in \mathbb{N}}$, which are called the terms, and a continuous map $q_{n,n+1}:X_{n+1}\rightarrow X_n$ for every $n\in \mathbb{N}$, which is called the bonding map, satisfying that $q_{n,m}=q_{n,l}\circ q_{l,m}$, where $n\leq l\leq m$. 
\begin{df} Let $(X_n,q_{n,n+1})$ be an inverse sequence of finite $T_0$ topological spaces. Let $\Pi_{n\in \mathbb{N}} X_n$ denote the Cartesian product of the topological spaces $X_i$ with the product topology. The inverse limit of $(X_n,q_{n,n+1})$ is a subspace of $\Pi_{n\in \mathbb{N}} X_n$ which consists of all points $x$ satisfying $\pi_n(x)=q_{n,m}(\pi_m(x))$ for every $m\geq n\in \mathbb{N}$, where $\pi_i:\Pi_{n\in \mathbb{N}} X_n\rightarrow X_i$ is the natural projection.
\end{df}
\begin{rem}
Notice that previous definitions can be given in a more abstract way for arbitrary categories, but we omit it for simplicity.
\end{rem}

Hence, the idea of this approach is the following: the bigger $n\in \mathbb{N}$ is, the better the term $X_n$ to approximate $X$ is. Indeed, the inverse limit of an inverse sequence where the index set is a finite totally ordered set is homeomorphic to the term indexed by the maximum. 
\begin{ex} Let us consider the Hawaiian earring, that is, 
$$\mathbb{H}=\bigcup_{n=1}^{\infty}\{(x,y)\in \mathbb{R}^2|(x-\frac{1}{n})^2+y^2=(\frac{1}{n})^2 \}.$$
We consider $X_n=\bigcup_{i=1}^n \{(x,y)\in \mathbb{R}^2|(x-\frac{1}{i})^2+y^2=(\frac{1}{i})^2 \}$. For every $n\in \mathbb{N}$ we have $X_{n}\subseteq X_{n+1}$. We also consider $p_{n,n+1}:X_{n+1}\rightarrow X_n$ given by $p_{n,n+1}(x,y)=(x,y)$ if $(x,y)\in X_n\subseteq X_{n+1}$ and $p_{n,n+1}(x,y)=(0,0)$ if $(x,y)\in X_{n+1}\setminus X_n$. We get an inverse sequence $(X_n,p_{n,n+1})$ satisfying that its inverse limit is homeomorphic to $\mathbb{H}$.  If $N$ denotes the totally ordered set $\{1,...,N\}$, then the inverse limit of $(X_n,p_{n,n+1},N)$ is $X_N$. The higher the value of $N$ is, the better the inverse limit of $(X_n,p_{n,n+1},N)$ approximates $\mathbb{H}$. We have represented this situation in Figure \ref{figure:anilloshawaianos}.
\begin{figure}[h]
\centering
\includegraphics[scale=0.75]{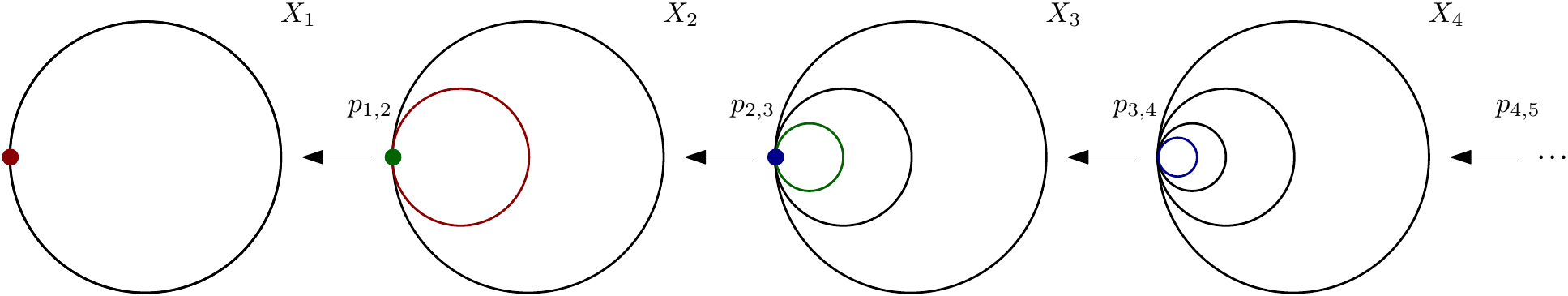}
\caption{Inverse sequence $(X_n,p_{n,n+1})$ and schematic representation of $p_{n,n+1}$.}\label{figure:anilloshawaianos}
\end{figure}
\end{ex}

As we mentioned before, polyhedra have been good candidates to get results of approximation. In \cite{clader2009inverse}, E. Clader proved that finite topological spaces can also be good candidates to approximate compact polyhedra. Namely, for every compact polyhedron $X$  there exists a natural inverse sequence of finite topological spaces such that its inverse limit contains a homeomorphic copy of $X$ which is a strong deformation retract. The idea of using finite topological spaces for this purpose was earlier suggested for instance in \cite{moron2008connectedness}, where it was introduced the so-called Main Construction and it was conjectured the General Principle. This principle states that the Main Construction can be used \textit{to extrapolate high dimensional topological properties of $X$, as in particular, the \v{C}ech homology groups in any dimension}. In \cite{mondejar2015hyperspaces,mondejar2020reconstruction}, it is proved a generalization of the result obtained in \cite{clader2009inverse} to compact metric spaces. In \cite{bilski2017inverse}, a similar result is obtained for topological spaces satisfying that are locally compact, paracompact and Hausdorff spaces, where Alexandroff spaces are considered. Throughout this manuscript, we will restrict our study to compact metric spaces and finite topological spaces. The main goal is to prove the General Principal using a different construction, which is more suitable for computational reasons.

We recall basic definitions, results and terminology for finite topological spaces. For a complete exposition about the theory of finite topological spaces see \cite{barmak2011algebraic} or \cite{may1966finite}. 

Given a finite $T_0$ topological space $X$ and $x\in X$. Let $U_x$ denote the intersection of every open set containing $x$. Analogously, let $F_x$ denote the intersection of every closed set containing $x$. Notice that $U_x$ is open and $F_x$ is closed.
\begin{df} Given a partially ordered set or poset $(X,\leq)$. A lower (upper) set $S\subset X$ is a set satisfying that if $x\in X$ and $y\leq x$ ($y\geq x$), then $y\in S$. 
\end{df}
It is not difficult to show the following two properties:
\begin{itemize}
\item For a finite poset $(X,\leq)$, the family of lower (upper) sets of $\leq$ is a $T_0$ topology on $X$, that makes $X$ a finite $T_0$ topological space.
\item For a finite $T_0$ topological space, the relation $x\leq_\tau$ y if and only if $U_x\subseteq U_y$ ($U_y\subseteq U_x$) is a partial order on $X$.
\end{itemize}
The partial order given in the second property is called the natural order, while the partial order given in parenthesis is called the opposite order.

A map $f:X\rightarrow Y$ between two posets is order-preserving if for every $x\leq x'$ in $X$, then $f(x)\leq f(x')$ in $Y$. It is easy to get the following proposition.
\begin{prop}\label{prop:continuidadpreservaorden} Let $f:X\rightarrow Y$ be a map between two finite $T_0$ topological spaces. Then $f$ is a continuous map if and only if $f$ is order-preserving.
\end{prop}
From this, we deduce the following theorem. 
\begin{thm}\label{thm:TeoremaAlexandroff} The category of finite $T_0$ topological spaces and the category of finite posets are isomorphic.
\end{thm}
Consequently, finite $T_0$ topological spaces and finite partially ordered sets can be seen as the same object from two different perspectives. From now on, every finite topological space satisfies the $T_0$ separation axiom. We will treat finite topological spaces and partially ordered sets as the same object without explicit mention.

Given a compact metric space $(X,d)$, we recall the Main Construction \cite{moron2008connectedness} or Finite Approximative Sequence (FAS) for $X$ \cite{mondejar2020reconstruction}.
\begin{df} Let $(X,d)$ be a compact metric space and let $\epsilon$ be a positive real number. A finite subset $A$ of $X$ is an $\epsilon$-approximation of $X$ if for every $x\in X$ there exists $a\in A$ satisfying that $d(x,a)<\epsilon$. 
\end{df}
Let $A$ be an $\epsilon$-approximation of $X$ and let $\mathcal{U}_{\epsilon}(A)=\{C\subseteq A|diam(C)<\epsilon\}$, where $diam(C)$ denotes the diameter of $C$. Then $\mathcal{U}_{\epsilon}(A)$ is a finite poset. The partial order is given by the subset relation, that is, $C\leq D$ if and only if $C\subseteq D$.

It is simple to deduce the following. If $(X,d)$ is a compact metric space, then for every $\epsilon>0$ there exists an $\epsilon$-approximation $A$ of $X$.

\begin{lem}\label{lem:mainpreparacion} Let $(X,d)$ be a compact metric space. If $\epsilon$ is a positive real value and $A$ is an $\epsilon$-approximation of $X$, then for every $0<\epsilon'<\frac{\epsilon-\gamma}{2}$ and $\epsilon'$-approximation $A'$ of $X$ the map $p:\mathcal{U}_{2\epsilon'}(A')\rightarrow \mathcal{U}_{2\epsilon}(A)$ given by $p(C)=\bigcup_{c\in C}\{a\in A|d(c,a)=d(c,A)\}$ is continuous, where $\gamma=sup\{d(x,A)|x\in X\}$.
\end{lem}
As an immediate consequence of Lemma \ref{lem:mainpreparacion}, it can be obtained the so-called Main Construction or FAS for a compact metric space $X$.
\begin{prop}[Main Construction or FAS ] Let $(X,d)$ be a compact metric space. Then there exists an inverse sequence $(\mathcal{U}_{2\epsilon_n}(A_n),p_{n,n+1})$, where $\{\epsilon_n\}_{n\in \mathbb{N}}$ is a sequence of decreasing positive real values with $\lim_{n\rightarrow\infty}\epsilon_n=0$, $\{A_n\}_{n\in \mathbb{N}}$ is a sequence of $\epsilon_n$-approximations of $X$ and the bonding map $p_{n,n+1}:\mathcal{U}_{2\epsilon_{n+1}}(A_{n+1})\rightarrow \mathcal{U}_{2\epsilon_n}(A_n)$ is given by $p_{n,n+1}(C)=\bigcup_{c\in C}\{a\in A|d(c,a)=d(c,A)\}$.
\end{prop}

One important property of this inverse sequence relies on its inverse limit.
\begin{thm}\label{thm:MainConstructionREconstruction} Let $(X,d)$ be a compact metric space and let $(\mathcal{U}_{2\epsilon_n}(A_n),p_{n,n+1})$ be a FAS for $X$. Then the inverse limit of $(\mathcal{U}_{2\epsilon_n}(A_n),p_{n,n+1})$ contains a homeomorphic copy of $X$ which is a strong deformation retract.
\end{thm}

If we consider the opposite order on the finite $T_0$ topological spaces of a FAS for a compact metric space, then the inverse limit does not need to preserve the good properties that were obtained in Theorem \ref{thm:MainConstructionREconstruction}. From a set theoretical point of view, the inverse limit is the same independently of the partial order chosen. From a topological point of view, the topologies are different.

\begin{ex}\label{ex:FASintervalo} Let us consider the unit interval $I=[0,1]$. We consider the same FAS for $I$ that was chosen in \cite[Example 4]{mondejar2015hyperspaces}. Namely, $A_1=\{0\}$, $\epsilon_1=2$, $\epsilon_n=\frac{1}{3^{2n-3}}$ and $A_n=\{\frac{k}{3^{2n-3}}|k=0,...,3^{2n-3}\}$ for every $n\in \mathbb{N}\setminus \{1\}$. Then $\mathcal{U}_{2\epsilon_n}(A_n)=A_n\cup \{ \{\frac{k}{3^{2n-3}}, \frac{k+1}{3^{2n-3}} \}|k=0,..., 3^{2n-3}-1\}$. Let $\varphi:\mathcal{I}^u\rightarrow I$ denote the map obtained in Theorem \ref{thm:MainConstructionREconstruction}, where $\mathcal{I}^u$ denotes the inverse limit of $(\mathcal{U}_{2\epsilon_n}(A_n),p_{n,n+1})$. Since $C=(C_1,C_2,C_3,...,C_n,...)\in \mathcal{I}^u$ can be seen as a sequence in the hyperspace of $I$, denoted by $2^I$, with the Hausdorff distance, it follows that $\varphi$ is defined by sending $(C_1,C_2,C_3,...)$ to its convergent point $\{x \}\in 2^I$ where $x\in I$. If $x\in \bigcup_{n\in \mathbb{N}}A_n$, then $\varphi^{-1}(x)$ has cardinality one. If $x\in I\setminus{ \bigcup_{n\in \mathbb{N}} A_n}$, then $\varphi(x)^{-1}=\{C,D,X\}$, where $C_i,D_i\subset X_i$ for every $i$. For a complete exposition of the previous assertion, see \cite[Chapter 3]{mondejar2015hyperspaces}. If we consider the opposite partial order in every term of the inverse sequence and $x\in I\setminus{ \bigcup_{n\in \mathbb{N}} A_n}$, then $X_i<_o  C_i$ and $X_i<_o D_i$ for every $i$, where $\varphi(x)^{-1}=\{C,D,X\}$. Let $\mathcal{I}^o$ denote the inverse limit of the inverse sequence of finite topological spaces with the opposite partial order. Then, the identity map $id:\mathcal{I}^o\rightarrow \mathcal{I}^u$ is not continuous. We argue by contradiction. Consider $x=\frac{1}{2}$. Then every open neighborhood $U$ of $C\in \mathcal{I}^o$ contains $X$ because $X_i<_o C_i$ for every $i$. On the other hand, consider $V=U_{C_1}\times U_{C_2}\times \cdots \times U_{C_n}\times \mathcal{U}_{2\epsilon_{n+1}(A_{n+1})}\times \mathcal{U}_{2\epsilon_{n+2}(A_{n+2})}\times \cdots$, where $U_{C_i}$ denotes the minimal open neighborhood of $C_i\in \mathcal{U}_{2\epsilon_i(A_i)}$. We have that $V$ is an open neighborhood of $id(C)$ and does not contain $X$ because $X_i> C_i$ for every $i$, which entails a contradiction. In Figure \ref{fig:differentInverseLimits} we have a schematic description of the situation described above.  

\begin{figure}[h]
\centering
\includegraphics[scale=1.15]{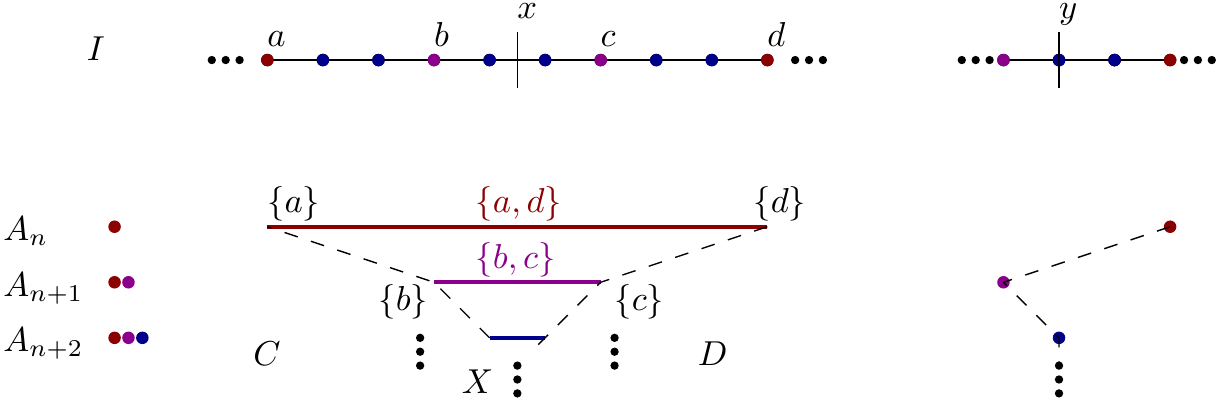}
\caption{Schematic description of $\varphi^{-1}(x)$ and $\varphi^{-1}(y)$, where $x\in I\setminus \{\bigcup_{n\in \mathbb{N}}A_n \}$ and $y\in \bigcup_{n\in \mathbb{N}}A_n$.}\label{fig:differentInverseLimits}
\end{figure}
\end{ex}

If we want to keep a similar result changing the partial order of the terms, then we need to get a different construction. Moreover, constructions of finite topological spaces involving the opposite partial order have been used recently to find applications to the study of dynamical systems, see for example \cite{lipinski2019conley}. The inverse sequence of finite topological spaces obtained in \cite{clader2009inverse} also uses the opposite order.

The organization of the paper is as follows. In Section \ref{section:constructionFASO} we construct an analogous of the Main Construction for a compact metric space using the opposite partial order. This construction has more advantages from a computational viewpoint. The different stages are implemented in classical algorithms or algorithms used to calculate persistent homology. A study about this issue is done in Section \ref{section:algoritmosComparativa} and two computational examples are also provided. In Section \ref{sec:inverseLimitProperties} the properties of the inverse limit of our inverse sequence are studied. It is proved that the inverse limit reconstructs the homotopy type. Then a result of uniqueness is given in Section \ref{sec:uniqueness}. This result can be seen as a sort of robustness. We also give an alternative inverse sequence that reconstructs algebraic invariants. This inverse sequence is more suitable for computational reasons and answers positively the General Principle \cite{moron2008connectedness}. In addition, we study the relations of our inverse sequences with the Main Construction and the construction given in \cite{clader2009inverse}. For completeness we have included at the end of this manuscript a brief appendix that contains basic definitions and results about pro-categories.

\section{Finite approximative sequences with the opposite order}\label{section:constructionFASO}
In this section, given a compact metric space $(X,d)$, we construct an inverse sequence of finite $T_0$ topological spaces using the opposite order instead of the natural order used in \cite{moron2008connectedness} and \cite{mondejar2020reconstruction}. From now on, if there is no explicit mention of the partial order considered on $\mathcal{U}_{4\epsilon}(A)$, where $\epsilon$ is a positive real value and $A$ is an $\epsilon$-approximation of $X$, then it is considered the partial order given as follows: $C\leq D$ if and only $D\subseteq C$. Let $\mathcal{B}(x,\epsilon)$ denote the open ball of center $x\in X$ and radius $\epsilon$.
\begin{lem}\label{lem:FASOpreparation} Given a compact metric space $(X,d)$. If $\epsilon$ is a positive real value and $A$ is an $\epsilon$-approximation of $X$, then for every $\epsilon'<\frac{\epsilon-\gamma}{2}$ and every $\epsilon'$-approximation $A'$ of $X$ the map $q:\mathcal{U}_{4\epsilon'}(A')\rightarrow \mathcal{U}_{4\epsilon}(A)$ given by $q(C)=\bigcup_{x\in C}\mathcal{B}(x,\epsilon)\cap A$ is well-defined and continuous, where $\gamma =sup\{ d(x,A)|x\in X\}$. 
\end{lem}
\begin{proof}
By compactness, $\gamma$ exists. We check that $q$ is well-defined. Let us take $C\in \mathcal{U}_{4\epsilon'}(A')$, which implies that $diam(C)<4\epsilon'$. If $x,y\in q(C)$, then there exist $c_x,c_y\in C$ satisfying that $x\in \mathcal{B}(c_x,\epsilon)$ and $y\in \mathcal{B}(c_y,\epsilon)$. Therefore, we have
$$d(x,y)\leq d(x,c_x)+d(c_x,c_y)+d(c_y,y)<\epsilon+4\epsilon'+\epsilon<2\epsilon+2(\epsilon-\gamma)<4\epsilon, $$
which implies that $diam(q(C))<4\epsilon$. The continuity of $q$ follows trivially.
\end{proof}
\begin{rem}\label{rem:Apaño} For simplicity, in Lemma \ref{lem:FASOpreparation}, we can consider $\epsilon'<\frac{\epsilon}{2}$ and the result also holds true.
\end{rem}

From here, we can get the desired construction.
\begin{thm}\label{thm:FASOconstruction} Let $(X,d)$ be a compact metric space. Then there exists an inverse sequence $(\mathcal{U}_{4\epsilon_n}(A_n),q_{n,n+1})$, where $\{\epsilon_n\}_{n\in \mathbb{N}}$ is a sequence of decreasing positive real values satisfying that $\lim_{n\rightarrow \infty}\epsilon_n=0$, $\{A_n \}_{n\in\mathbb{N}}$ is a sequence of $\epsilon_n$-approximations of $X$ and the map $q_{n,n+1}:\mathcal{U}_{4\epsilon_{n+1}}(A_{n+1})\rightarrow \mathcal{U}_{4\epsilon_n}(A_n)$ is given by $q_{n,n+1}(C)=\bigcup_{x\in C}\mathcal{B}(x,\epsilon_n)\cap A_n$.
\end{thm}
\begin{proof}
Let $\epsilon_1>diam(X)$. Then $A_1$ can be taken as $A_1=\{ a\}$ for some $a\in X$ and $\gamma_1=\sup\{d(x,A_1)|x\in X\}$. Applying Lemma \ref{lem:FASOpreparation}, we can obtain $\epsilon_2<\frac{\epsilon_1-\gamma_1}{2}$, an $\epsilon_2$-approximation $A_2$ of $X$ and a continuous map $q_{1,2}:\mathcal{U}_{4\epsilon_2}:(A_2)\rightarrow \mathcal{U}_{4\epsilon_1}(A_1)$. We only need to repeat this method inductively to conclude.
\end{proof}

Given a compact metric space $(X,d)$. The inverse sequence $(\mathcal{U}_{4\epsilon_n}(A_n),q_{n,n+1})$ obtained in the proof of Theorem \ref{thm:FASOconstruction} is called Finite Approximative Sequence with Opposite order (FASO). For simplicity, when there is no confusion, $(\mathcal{U}_n,q_{n,n+1})$ denotes $(\mathcal{U}_{4\epsilon_n}(A_n),q_{n,n+1})$.

\begin{ex}\label{ex:FASoS1primeraparte} We construct a FASO for the unit circle $S^1$. We consider the unit circle in the complex plane with the geodesic distance, $S^1=\{ z\in \mathbb{C}| |z|=1\}$. We get a FASO for $S^1$ by steps.

\underline{\textit{Step 1.}} We consider $\epsilon_1=3\pi>diam(S^1)=\pi$ and $A_1=\{e^{2\pi i} \}$, which is clearly an $\epsilon_1$-approximation of $S^1$. Then, $\gamma_1=\pi$ and $\mathcal{U}_1=A_1$.

 \underline{\textit{Step 2.}} We consider $\epsilon_2=\frac{\pi}{2}<\frac{3\pi-\pi}{2}$ and $A_2=\{a_k^2 =e^{\frac{2\pi ki}{4}}|k=0,1,2,3\}$, which is clearly an $\epsilon_2$-approximation of $S^1$. Then, $\gamma_2=\frac{\pi}{4}$ and $\mathcal{U}_2=\{2^{A_2} \}$, where $2^{A_n}$ denotes the power set of $S$ minus the empty set. 

\underline{\textit{Step 3.}} We consider $\epsilon_3=\frac{\pi}{16}<\frac{\pi}{8}$ and $A_3=\{a_k^3=e^{\frac{2\pi k i }{32}} | k=0,1,...,31\}$, which is clearly an $\epsilon_3$-approximation. Then, $\gamma_3=\frac{\pi}{32}$ and $\mathcal{U}_3=\{ 2^{a_i^3,a_{i+1}^3,a_{i+2}^3,a_{i+3}^3}\}_{i=0,...,31}$, where the subindices are considered modulo $32$ and $2^{a_i^3,a_{i+1}^3,a_{i+2}^3,a_{i+3}^3}$ denotes the power set of $\{a_i^3,a_{i+1}^3,a_{i+2}^3,a_{i+3}^3\}$ minus the empty set. The last statement is true due to the fact that $d(a_i^3,a_{i+1}^3)=\frac{ \pi}{16}$ for every $i,i+1$ modulo $32$ and $4\epsilon_3=\frac{\pi}{4}$. 

\underline{\textit{Step $n$.}} We consider $\epsilon_n=\frac{\pi}{2^{3n-5}}$ and $A_n=\{a^n_k=e^{\frac{2\pi k i }{2^{3n-4}}}|k=0,...,2^{3n-4}-1 \}$, which is clearly an $\epsilon_n$-approximation of $S^1$. Therefore, we obtain that $\gamma_n=\frac{\pi}{2^{3n-4}}$ and $\mathcal{U}_n=\{ 2^{a_i^n,a_{i+1}^n,a_{i+2}^n,a_{i+3}^n}\}_{i=0,...,2^{3n-4}-1 }$, where the subindices are considered modulo $2^{3n-4}$ and $2^{a_i^n,a_{i+1}^n,a_{i+2}^n,a_{i+3}^n}$ denotes the power set of $\{a_i^n,a_{i+1}^n,a_{i+2}^n,a_{i+3}^n\}$ minus the empty set. The last statement is true due to the fact that $d(a_i^n,a_{i+1}^n)=\frac{\pi}{2^{3n-5}}$ for every $i,i+1$ modulo $2^{3n-4}$ and $4\epsilon_n=\frac{\pi}{2^{3n-7}}$.

For a schematic representation of the minimal points of $\mathcal{U}_1$, $\mathcal{U}_2$ and $\mathcal{U}_3$, see Figure \ref{fig:CirculoFASOminimales}. Each arc represents a minimal point. Red, green and blue arcs represent the minimal points of $\mathcal{U}_1$, $\mathcal{U}_2$ and $\mathcal{U}_3$, respectively.

\begin{figure}[h]
\centering
\includegraphics[scale=0.7]{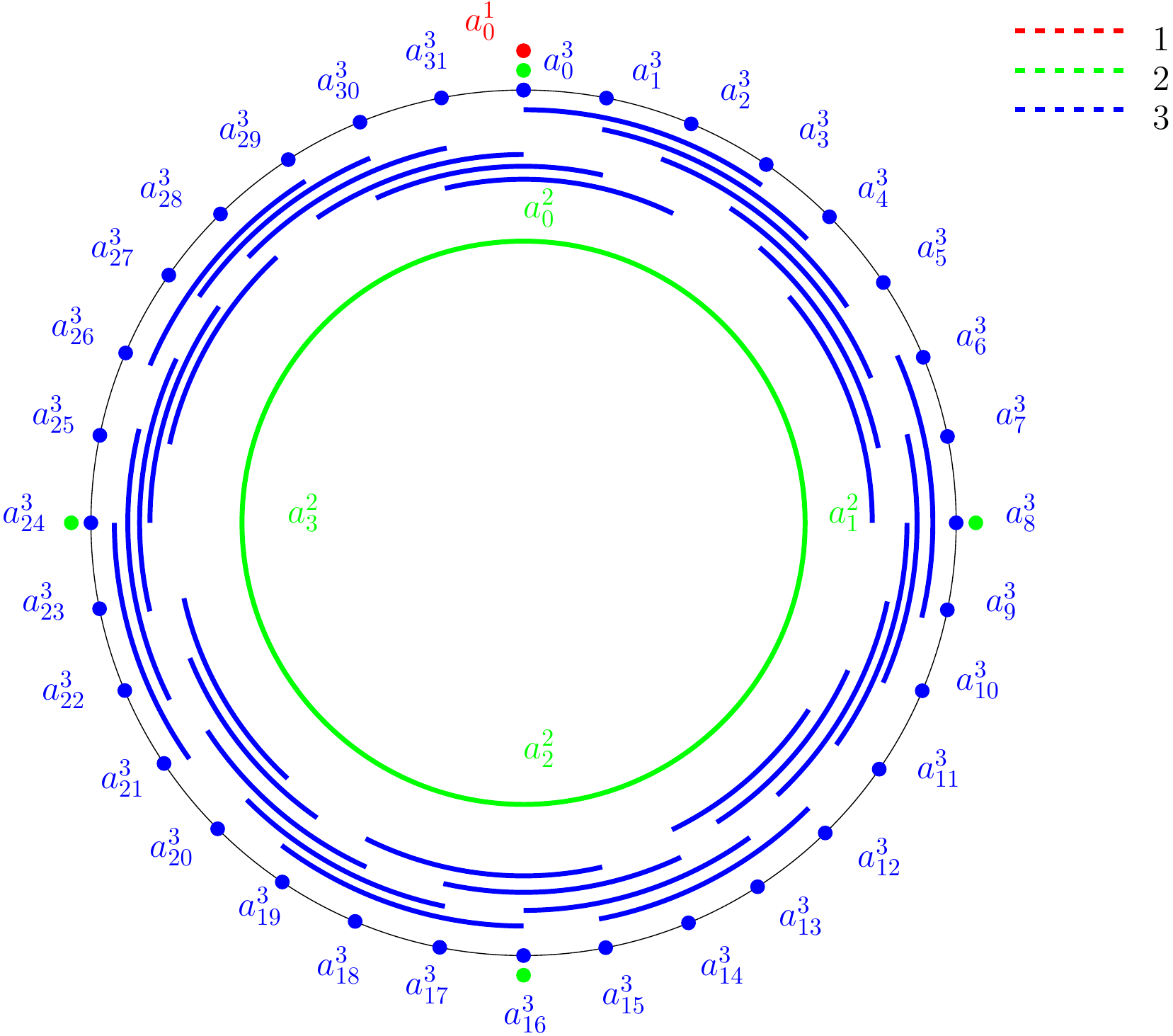}
\caption{Diagram of the minimal points of $\mathcal{U}_1$, $\mathcal{U}_2$ and $\mathcal{U}_3$ in $S^1$. } \label{fig:CirculoFASOminimales}
\end{figure}
\end{ex}

Given a compact metric space $(X,d)$ and a FASO $(\mathcal{U}_n,q_{n,n+1})$ for $X$, there is natural map $q_{n}:X\rightarrow \mathcal{U}_n$ given by $q_n(x)=\mathcal{B}(x,\epsilon_n)\cap A_n$ for every $n\in \mathbb{N}$. 

\begin{prop}\label{prop:contprojection} Given a compact metric space $(X,d)$ and a FASO $(\mathcal{U}_n, q_{n,n+1})$ for $X$. The following diagram commutes up to homotopy for every $n\in \mathbb{N}$.
\[
  \begin{tikzcd}[row sep=large,column sep=huge]
& X \arrow[dl,"q_n"'] \arrow{dr}{q_{n+1}} &\\
\mathcal{U}_{4\epsilon_n}(A_n) &  &  \mathcal{U}_{4\epsilon_{n+1}}(A_{n+1}) \arrow{ll}{q_{n,n+1}}
    \end{tikzcd} 
\]
\end{prop}
\begin{proof}
We prove that $q_n$ is continuous and well-defined for every $n\in \mathbb{N}$. If $x\in X$, then the diameter of $q_n(x)$ is less than $2\epsilon_n$. This implies that $q_n(x)\in \mathcal{U}_{4\epsilon_n}(A_n)$. Now, we prove the continuity of $q_n$. We consider $\gamma=\max\{d(x,b)|b\in q_n(x) \}$ and $0<\delta<\epsilon_n-\gamma$. We have that $\gamma$ is well-defined since $q_n(x)$ is a finite set. For every $y\in \mathcal{B}(x,\delta)$ we have that $q_n(x)\subseteq q_n(y)$. We prove the last assertion. If $b\in q_n(x)$, then we get $d(x,b)\leq \gamma$. Therefore, $$d(y,b)<d(y,x)+d(x,b)<\epsilon_n-\gamma+\gamma=\epsilon_n.$$
From this, the continuity of $q_n$ follows easily since we have that $q_n(\mathcal{B}(x,\delta))\subseteq U_{q_n(x)}$, where $U_{q_n(x)}$ denotes the minimal open neighborhood of $q_n(x)$.

We define $h:X\rightarrow \mathcal{U}_n$ given by $h(x)=q_n(x)\cup q_{n,n+1}(q_{n+1}(x))$. We prove that $h(x)$ has diameter less than $4\epsilon_n$. Let us take $a\in q_n(x)$ and $b\in q_{n,n+1}(q_{n+1}(x))$. By construction, $d(a,x)<\epsilon_n$ and there exist $C\in \mathcal{U}_{4\epsilon_{n+1}}(A_{n+1})$ and $c\in C$ such that $c\in q_{n+1}(x)=C$ and $b\in q_{n,n+1}(c)$, which implies that $d(x,c)<\epsilon_{n+1}<\frac{\epsilon_n}{2}$ and $d(c,b)<2\epsilon_n$. We have
$$d(a,b)<d(a,x)+d(x,c)+d(c,b)<\epsilon_n+\frac{\epsilon_n}{2}+2\epsilon_n<4\epsilon_n. $$
Therefore, $h=q_n \cup q_{n,n+1}\circ q_{n+1}:X\rightarrow\mathcal{U}_{4\epsilon_n}(A_n)$ is well-defined. The continuity of $h$ follows trivially. In addition, for every $x\in X$ we get that $q_n(x),q_{n,n+1}(q_{n+1}(x))\subseteq q_n(x)\cup q_{n,n+1}(q_{n+1}(x))$. We show that the diagram commutes up to homotopy. We only prove one of the two homotopies because the other one is similar. We define $H:X\times [0,1]\rightarrow \mathcal{U}_{4\epsilon_n}(A_n)$ given by
\[
  H(x,t) =
  \begin{cases}
                               q_n(x) \cup q_{n,n+1}(q_{n+1}(x)) & \text{if $t\in[0,1)$}      \\
                       q_n(x) & \text{if $t=1$}.            
  \end{cases}
\]
It suffices to verify the continuity of $H$ at $t=1$. If $(x,1)\in X\times[0,1]$, then we consider the minimal open neighborhood of $H(x,1)$, that is, $U_{H(x,1)}$. By the continuity of $q_n$, there exists an open neighborhood $V$ of $x$ with $H(V,1)\subset U_{H(x,1)}$. By construction, if $t\neq 1$, then we have $H(V,t)=h(V)\supset H(V,1)=q_n(V)$ so $h(y)\in U_{q_{n}(x)}$ for every $y\in V$. Therefore, $V\times  [0,1]$ is an open neighborhood of $(x,1)$ in $X\times [0,1]$ satisfying $H(V,[0,1])\subset U_{H(x,1)}$, which implies the continuity of $H$ at $(x,1)$.
\end{proof} %

Given a compact metric space $(X,d)$ and a FASO $(\mathcal{U}_n,q_{n,n+1})$ for $X$. If we consider the other possible partial order defined on every term of the inverse sequence $(\mathcal{U}_n,q_{n,n+1})$, then the bonding maps are also continuous, but Proposition \ref{prop:contprojection} does not hold true. This is due to the continuity of the map $q_n:X\rightarrow \mathcal{U}_n$. In fact, we have the following result.
\begin{prop} Given a connected compact metric space $(X,d)$ and a finite topological space $Y$. If $f:X\rightarrow Y$ is continuous and $f$ is also continuous when it is considered the other possible order on $Y$, then $f$ is the constant map.
\end{prop}

\begin{proof}
Let us consider $x\in X$ and the minimal open neighborhood containing $f(x)$ for the natural order and opposite order, that is, $U_{f(x)}$ and $U_{f(x)}'$ respectively. By the continuity of $f$, there exist open sets $V_x$ and $W_x$ containing $x$ such that $f(V_x)\subseteq U_{f(x)}$ and $f(W_x)\subseteq U_{f(x)}'$. Therefore, $f(V_x\cap W_x)\subseteq U_{f(x)}\cap U_{f(x)}'=\{f(x) \}$, which implies that $f$ is a locally constant map. Since $X$ is connected, it follows that $f$ is a constant map.
\end{proof}
\section{Properties of the inverse limit of a FASO for a compact metric space}\label{sec:inverseLimitProperties}

Given a compact metric space $(X,d)$ and a FASO $(\mathcal{U}_n,q_{n,n+1})$ for $X$. We study properties of the inverse limit of $(\mathcal{U}_n,q_{n,n+1})$, denoted by $\mathcal{X}$. Firstly, we prove that $\mathcal{X}$ is non-empty. Despite the fact that this result can be deduced from \cite[Theorem 2]{stone1979inverse}, we prefer to describe specific elements of $\mathcal{X}$.

For every $x\in X$ we consider 
$$X_*^n=\bigcup_{m>n} q_{n,m}(A_m(x)), $$
where $A_m(x)=\{a\in A_m|d(x,a)=d(x,A_m) \}$. The sequence $\{X_*^n\}_{n\in \mathbb{N}}$ is a candidate to be an element of $\mathcal{X}$.

\begin{prop}\label{prop:X*finite} 
If $x\in X$, then $X_*^n\in \mathcal{U}_n$ for every $n\in \mathbb{N}$.
\end{prop}
\begin{proof}
If $x\in X$, then $X_*^n\subseteq \mathcal{B}(x,2\epsilon_n)\cap A_n$ for every $n\in \mathbb{N}$. We prove the last assertion verifying that $q_{n,m}(A_m(x))\subset \mathcal{B}(x,2\epsilon_n)\cap A_n$ for all $m>n$. As a consequence, it can be deduced that $diam(X_*^n)<4\epsilon_n$ and $X_*^n\in \mathcal{U}_{4\epsilon_n}(A_n)$.

If $a_n\in q_{n,m}(A_m(x))$, then there exists a sequence $\{a_t\}_{n\leq t\leq m}$ with $a_t\in A_t$ and $a_t\in q_{t,t+1}(a_{t+1})$, so $d(a_t,a_{t+1})<\epsilon_t$. In addition, $a_m\in A_m(x)$, which means $d(a_m,x)<\epsilon_m$. Therefore,
\begin{align*}
d(a_n,x)& <d(a_n,a_{n+1})+d(a_{n+1},a_{n+2})+\dots+d(a_{m-1},a_m)+d(a_m,x)<\\
&<\epsilon_n+\epsilon_{n+1}+\dots +\epsilon_{m-1}+\epsilon_m <\epsilon_n+\frac{\epsilon_{n}}{2}+\dots + \frac{\epsilon_{n}}{2^{m-1-n}}+\frac{\epsilon_{n}}{2^{m-n}}=\\
&=\epsilon_n(1+\sum_{i=1}^{m-n}\frac{1}{2^i})<2\epsilon_n.
\end{align*}
\end{proof}

The idea of the following lemmas is to show that $X_*^n$, which is an infinite union of sets, stabilizes for every $n\in \mathbb{N}$.

\begin{lem}\label{lem:contenidos} If $x\in X$, then $A_n(x)\subset q_{n,n+1}(A_{n+1}(x))$ for every $n\in \mathbb{N}$.
\end{lem}
\begin{proof}
If $a_n\in A_n(x)$ and $a_{n+1}\in A_{n+1}(x)$, then we have $d(a_n,x)\leq \gamma_n$ and $d(a_{n+1},x)<\epsilon_{n+1}$, respectively. We have the following relation
$$d(a_n,a_{n+1})<d(a_n,x)+d(x,a_{n+1})<\gamma_n+\epsilon_{n+1}<\gamma_n+\frac{\epsilon_n-\gamma_n}{2}=\frac{\epsilon_n}{2}+\frac{\gamma_n}{2}<\epsilon_n, $$
so $a_n\in \mathcal{B}(a_{n+1},\epsilon_n)\cap A_n=q_{n,n+1}(a_{n+1})$.
\end{proof}

\begin{lem}\label{lem:qContainsq} If $x\in X$, then $q_{n,m}(A_m(x))\subset q_{n,m+1}(A_{m+1}(x))$ for every $m>n$.
\end{lem}
\begin{proof}
We know that $q_{n,m+1}(A_{m+1}(x))=q_{n,m}(q_{m,m+1}(A_{m+1}(x)))$. By Lemma \ref{lem:contenidos}, we have that $A_m(x)\subset q_{m,m+1}(A_{m+1}(x))$. On the other hand, $q_{n,m}$ is a continuous map between finite topological spaces, so $q_{n,m}$ preserves the subset relation. If we apply $q_{n,m}$ to $A_m(x)\subset q_{m,m+1}(A_{m+1}(x))$, then we get $q_{n,m}(A_m(x))\subseteq q_{n,m}(q_{m,m+1}(A_{m+1}(x)))=q_{n,m+1}(A_{m+1}(x))$.
\end{proof}

\begin{prop}\label{prop:estabiliza} If $x\in X$, then for every $n\in \mathbb{N}$ there exists $n_*>n$ such that for all $m\geq n_*$  $X_*^n=q_{n,m}(A_m(x))$.
\end{prop}
\begin{proof}
By Proposition \ref{prop:X*finite} we know that $X_*^n\subset \mathcal{B}(x,2\epsilon_n)\cap A_n$, so $X_*^n$ is a finite set. By Lemma \ref{lem:qContainsq}, if $a\in X_*^n$, then there exists $n_a\in \mathbb{N}$ such that $a\in q_{n,m}(A_m(x))$ for every $ m\geq n_a$. We consider 
$$n_*=max\{n_a| a\in X_*^n \quad and \quad a\in q_{n,m}(A_{m}(x))  \quad \text{with} \quad m\geq n_a \}, $$
where $n_*$ is well-defined since  $X_*^n$ is a finite set. From here, it follows the desired result.
\end{proof}

Finally, we prove that $\{ X_*^n\}_{n\in \mathbb{N}}$ is an element of the inverse limit $\mathcal{X}$. Then, we also prove a connection between the elements of $\mathcal{X}$ and $X$.
\begin{prop}\label{prop:elementoslimiteinverso} If $x\in X$, then $\{X_*^n \}_{n\in \mathbb{N}}\in \mathcal{X}$. 
\end{prop}
\begin{proof}
It is only necessary to check that $q_{n,n+1}(X_*^{n+1})=X_*^n$, the general case follows inductively. By Proposition \ref{prop:estabiliza}, if $s>n_*,(n+1)_*$, then $q_{n,s}(A_s(x))=X_*^n$ and $q_{n+1,s}(A_{s}(x))=X_*^{n+1}$. Thus,
$$q_{n,n+1}(X_*^{n+1})=q_{n,n+1}(q_{n+1,s}(A_s(x)))=q_{n,s}(A_s(x))=X_*^n. $$
\end{proof}


%

We recall some basic definitions and properties that we need. Given a compact metric space $(X,d)$, $2^X=\{C\subseteq X| C$ is non-empty and closed $\}$ is called the hyperspace of $X$. There is a natural metric that can be defined on $2^X$, the Hausdorff metric $d_H$. If $C,D\in 2^X$, then the Hausdorff metric is defined as follows:
$$d_H(C,D)=\inf \{\epsilon>0|C\subseteq \mathcal{B}(D,\epsilon), \ D\subseteq \mathcal{B}(C,\epsilon) \}, $$
where $\mathcal{B}(C,\epsilon)$ and $\mathcal{B}(D,\epsilon)$ denote the generalize ball of radius $\epsilon$, that is, if $C\in 2^X$, then $\mathcal{B}(C,\epsilon)=\{x\in X|d(x,C)<\epsilon \}$. The hyperspace of $X$ with the Hausdorff metric is a compact metric space. We recollect in the following proposition some properties of the Hausdorff metric.
\begin{prop}\label{prop:propiedadesmetricahausdorff} Let $(X,d)$ be a compact metric space and let $(2^X,d_H)$ be the hyperspace of $X$ with the Hausdorff metric. 
\begin{itemize}
\item $d_H(x,y)=d(x,y)$ if $x,y\in X$.
\item $d_H(x,C)=\sup \{d(x,c)|c\in C \}\geq \inf \{ d(x,c)|c\in C\}=d(x,C)$ if $x\in X$ and $C\in 2^X $.
\item $d_H(x,D)\leq d_H(x,C)$ but $d(x,D)\geq d(x,C)$ if $x\in X$ and $C,D\in 2^X$, where $D\subseteq C$.
\end{itemize}
\end{prop}

\begin{prop}\label{prop:convergencia} If $\{ C_n \}_{n\in \mathbb{N}}\in \mathcal{X}$, then $\{ C_n \}_{n\in \mathbb{N}}$ is a Cauchy sequence in $(2^X,d_H)$ that converges to $\{ x\}\in 2^X$ for some $x\in X$. Moreover, $d_H(x,C_n)<2\epsilon_n$ for every $n\in \mathbb{N}$.
\end{prop}
\begin{proof}
Firstly, we check that $d_H(C_n,C_m)<2\epsilon_n-\frac{\gamma_n}{2}$ for every $n,m\in \mathbb{N}$ satisfying $m\geq n$. If $c_n\in C_n$, then there exists a sequence $\{c_t \}_{n\leq t\leq m}$ with $c_t\in A_t$ and $c_t\in q_{t,t+1}(c_{t+1})$, so $d(c_t,c_{t+1})<\epsilon_t$. Then,
\begin{align*}
d(c_n,c_m)&<d(c_n,c_{n+1})+\dots+d(c_{m-1},c_m)<\epsilon_n+\dots +\epsilon_{m-1}<\\ 
& <\epsilon_n+\frac{\epsilon_n-\gamma_n}{2}+\dots + \frac{\epsilon_n-\gamma_n}{2^{m-1-n}}=\epsilon_n(1+\sum_{i=1}^{m-1-n}\frac{1}{2^i})-\gamma_n\sum_{i=1}^{m-1-n} \frac{1}{2^i}\\ &<2\epsilon_n-\frac{\gamma_n}{2},
\end{align*}
which implies that $C_n\subset \mathcal{B}(C_m,2\epsilon_n-\frac{\gamma_n}{2})$. If $c_m\in C_m$, then we can repeat the same argument to show that $C_m\subset \mathcal{B}(C_n,2\epsilon_n-\frac{\gamma_n}{2})$. Therefore, for every $\epsilon>0$ there exists $s\in \mathbb{N}$ such that for every $n,m>s$ we have $d_H(C_n,C_m)<\epsilon$. It is only necessary to consider $s$ satisfying that $2\epsilon_s-\frac{\gamma_s}{2}<\epsilon$. Hence, $\{C_n\}_{n \in \mathbb{N}}$ is a Cauchy sequence in a compact metric space $2^X$, $\{C_n\}_{n\in \mathbb{N}}$ converges to an element $C\in 2^X_H$. It is important to recall that $diam(C_n)\leq 4\epsilon_n$ because $C_n\in \mathcal{U}_{4\epsilon_n}(A_n)$. Due to the fact of the continuity of the diameter function regarding to the Hausdorff metric, we have 
$$diam(C)=diam(\lim_{n\rightarrow \infty} (C_n))=\lim_{n\rightarrow \infty}(diam (C_n))< \lim_{n\rightarrow \infty} 4\epsilon_n=0. $$
Thus, $C=\{x\}$ for some $x\in X$.

We have shown that $d_H(C_n,C_m)<2\epsilon_n-\frac{\gamma_n}{2}$ for every $m,n\in \mathbb{N}$ satisfying $m\geq n$. $\{C_n\}_{n\in \mathbb{N}}$ is a Cauchy sequence that converges to $\{ x\}$. Then for $\frac{\gamma_n}{2}$ there exists $n_0>n$ such that for every $m>n_0$ we get $d_H(x,C_m)<\frac{\gamma_n}{2}$. Therefore,
$$d_H(x,C_n)<d_H(x,C_m)+d_H(C_m,C_n)<\frac{\gamma_n}{2}+2\epsilon_n-\frac{\gamma_n}{2}=2\epsilon_n. $$
\end{proof}

\begin{prop}\label{prop:surjectivity} If $x\in X$, then $\{X_*^n\}_{n \in \mathbb{N}}\in \mathcal{X}$ converges to $\{x\}$ in $2^X_H$.
\end{prop}
\begin{proof}
It is an immediate consequence of Proposition \ref{prop:X*finite} and Proposition \ref{prop:convergencia}.
\end{proof}

Furthermore, we also get that $\{ X_*^n \}_{n \in \mathbb{N}}$ is somehow minimal with respect to the elements of $\mathcal{X}$ that converge to the same point.

\begin{prop}\label{prop:minimal} If $\{C_n\}_{n \in \mathbb{N}}\in \mathcal{X}$ converges to $\{x\}$, then $X_*^n\subseteq C_n$ for every $n\in \mathbb{N}$.
\end{prop}
\begin{proof}
We prove that $A_n(x)\subset C_n$ for all $n\in \mathbb{N}$. We know that $q_{n,n+1}(C_{n+1})=C_n$. If $c_{n+1}\in C_{n+1}$, then we have that $d_H(x,C_{n+1})<2\epsilon_{n+1}$ by Proposition \ref{prop:convergencia}. Therefore, $d(x,c_{n+1})<2\epsilon_{n+1}$. In addition, for every $a_n\in A_n(x)$, we get $d(x,a_n)\leq \gamma_n$. We have
$$d(a_n,c_{n+1})<d(a_n,x)+d(x,c_{n+1})<\gamma_n+2\epsilon_{n+1}<\gamma_n+\epsilon_n-\gamma_n=\epsilon_n, $$
so $a_n\in \mathcal{B}(c_{n+1},\epsilon_n)\cap A_n=q_{n,n+1}(c_{n+1})\subset C_n$. We can conclude that $A_n(x)\subset C_n$. We take $s>n_*$ where $n_*$ is given by Proposition \ref{prop:estabiliza}. Then $q_{n,s}(A_s(x))=X_*^n$. On the other hand, we have proved that $A_s(x)\subset C_s$. If we apply $q_{n,s}$ to the previous content, then we get the desired result because $q_{n,s}$ is a continuous map between finite topological spaces, which means that it preserves the subset relation,
$$X_*^n=q_{n,s}(A_s(x))\subset q_{n,s}(C_s)=C_n. $$
\end{proof}

Now, we can define a map $\varphi$ between the inverse limit of $(\mathcal{U}_n,q_{n,m})$ and $X$, $\varphi:\mathcal{X}\rightarrow X$. The map $\varphi$ sends each element of the inverse limit to its convergent point given by Proposition \ref{prop:convergencia}, that is, $\varphi(\{C_n \}_{n\in \mathbb{N}})=x$, where $\lim_{n\rightarrow\infty}C_n= \{x \}$.

\begin{prop} $\varphi:\mathcal{X}\rightarrow X$ is surjective and continuous.
\end{prop}
\begin{proof}
The surjectivity is given by the construction of $\{X_*^n\}_{n\in \mathbb{N}}$ and Proposition \ref{prop:surjectivity}, so it only remains to show the continuity. For each open neighborhood $V$ of $\varphi(\{C_n\}_{n\in \mathbb{N}})=x$, we can take $\delta>0$ such that $\mathcal{B}(x,\delta)\subset V$. Now, we consider the open neighborhood of $\{C_n\}_{n\in \mathbb{N}}$ given as follows:
$$W=(2_{C_1}\times 2_{C_2}\times \dots \times 2_{C_{n_0}}\times \mathcal{U}_{4\epsilon_{n_0+1}(A_{n_0+1})}\times \dots )\cap \mathcal{X}, $$
where $2_{C_t}=\{ D\in \mathcal{U}_{4\epsilon_t(A_t)}| C_t \subset D\}$  denotes the minimal open neighborhood of $C_t$ in $\mathcal{U}_{4\epsilon_t}(A_t)$. We consider $n_0$ satisfying that for every  $n\geq n_0$, it is obtained that  $\epsilon_n<\frac{\delta}{4}$. If $\{ D_n\}_{n \in \mathbb{N}}\in W$ with $\{D_n\}\rightarrow \{ y\}$, then we check that $y\in \mathcal{B}(x,\delta)$. By construction, $C_n \subset D_n$ for every $n\leq n_0$. By the second part of Proposition \ref{prop:convergencia} and the previous observation, we get
\begin{align*}
d(x,y)=d_H(x,y)&<d_H(x,C_{n_0})+d_H(C_{n_0},y)<d_H(x,C_{n_0})+d_H(D_{n_0}, y )< \\& <2\epsilon_{n_0}+2\epsilon_{n_0}<4\epsilon_{n_0}<\delta,
\end{align*}
where we are using the properties of the Hausdorff metric given in Proposition \ref{prop:propiedadesmetricahausdorff}.
\end{proof}

We can also define a map $\phi$ between $X$ and $\mathcal{X}$ given by the construction made at the beginning, that is, $\phi(x)=\{X_*^n\}_{n \in \mathbb{N}}$.
\begin{prop}\label{prop:phiInjective} $\phi$ is injective and continuous.
\end{prop}
\begin{proof}
We show the continuity of $\phi$. For each open neighborhood $V$ of $\phi(x)=\{X_*^n\}_{n \in \mathbb{N}}$ we can find an open neighborhood of the form 
$$W=(2_{X_*^1}\times 2_{X_*^2}\times \dots \times 2_{X_*^{n}}\times \mathcal{U}_{4\epsilon_{n+1}(A_{n+1})}\times \dots )\cap \mathcal{X} $$
such that $W\subset V$. By Proposition \ref{prop:estabiliza}, for every $X_*^n$ there exists $n_*$ satisfying that for every $s>n_*$, $q_{n,s}(A_s(x))=X_*^n$. We fix a value $s>n_*$. Now, we consider $\delta<\epsilon_s-\gamma_s-\epsilon_{s+1}$, where we have $\epsilon_s-\gamma_s-\epsilon_{s+1}>0$ because $\epsilon_{s+1}<\frac{\epsilon_s-\gamma_s}{2}$. The idea is to verify that for every $y\in \mathcal{B}(x,\delta)$, it is obtained $\phi(y)=\{Y_*^n \}_{n\in \mathbb{N}}\in W$. By the continuity of $q_{m,n}$ for all $m<n$, it is only necessary to check that $X_*^n\subset Y_*^n$ because $X_*^m=q_{m,n}(X_*^n)\subset q_{m,n}(Y_*^n)=Y_*^m$. From here, we would get $\phi(y)\in W$.

We check that $A_s(x)\subset q_{s,s+1}(A_{s+1}(y))$. If $a_{s+1}\in A_{s+1}(y)$, then we have $d(y,a_{s+1})<\epsilon_{s+1}$. On the other hand, if $b_s\in A_s(x)$, we know that $d(x,b_s)\leq\gamma_s$. Then,
$$d(b_s,a_{s+1})<d(b_s,x)+d(x,y)+d(y,a_{s+1})<\gamma_s+\delta+\epsilon_{s+1}<\gamma_s+\epsilon_s-\gamma_s-\epsilon_{s+1}+\epsilon_{s+1}=\epsilon_s, $$
which means that $b_s\in \mathcal{B}(a_{s+1},\epsilon_s)\cap A_s=q_{s,s+1}(a_{s+1})$, so $A_s(x)\subset q_{s,s+1}(A_{s+1}(y))$, as we wanted. If we apply $q_{n,s}$ to the previous content, then we get
$$X_*^n=q_{n,s}(A_s(x))\subset q_{n,s}(q_{s,s+1}(A_{s+1}(y)))=q_{n,s+1}(A_{s+1}(y))\subset Y_*^n=\bigcup_{m>n}q_{n,m}(A_m(y)). $$ 

Now, we prove the injectivity of $\phi$. If $x\neq y$ we have $d(x,y)>0$, then we take $n_0$ such that for all $n>n_0$, we get $\epsilon_n<\frac{d(x,y)}{16}$. By Proposition \ref{prop:X*finite}, we know that $X_*^n\subset \mathcal{B}(x,2\epsilon_n)$ and $Y_*^n\subset \mathcal{B}(y,2\epsilon_n)$. If $X_*^n\cap Y_*^n\neq \emptyset$, then we can take $a\in X_*^n\cap Y_*^n$. Hence, $d(a,x),d(a,y)<2\epsilon_n$ and we get a contradiction since
$$16\epsilon_n <d(x,y)<d(x,a)+d(a,y)<2\epsilon_n+2\epsilon_n<4\epsilon. $$
Thus,  $X_*^n\cap Y_*^n= \emptyset$ and we conclude that $\phi(x)\neq \phi(y)$.
\end{proof}

Let $\mathcal{X}_*$ denote $\phi(X)\subseteq\mathcal{X}$. We will prove that $\mathcal{X}_*$ is homeomorphic to $X$ and a strong deformation retract of $\mathcal{X}$. Firstly, we verify that $\mathcal{X}_*$ is a homeomorphic copy of $X$ in $\mathcal{X}$, but before enunciating this result we prove a property of $\mathcal{X}_*$ that will be used. 

\begin{prop}\label{prop:hausdorff} $\mathcal{X}_*$ is a Hausdorff space.
\end{prop}
\begin{proof}
Let us take $\phi(x)=\{X_*^n\}_{n\in \mathbb{N}}\neq \{Y_*^n \}_{n\in \mathbb{N}}=\phi(y)$, so $x\neq y$ and $d(x,y)>0$. We consider $n_0$ such that for every $n\geq n_0$ we get $\epsilon_n<\frac{d(x,y)}{16}$. Furthermore, we know that $X_*^n\cap Y_*^n=\emptyset$ for all $n\geq n_0$ by the proof of Proposition \ref{prop:phiInjective}. We consider the following open neighborhoods for $\{ X_*^n\}_{n\in \mathbb{N}}$ and $\{ Y_*^n\}_{n\in \mathbb{N}}$ respectively
\begin{align*}
V_1&=(2_{X_*^1}\times 2_{X_*^2}\times \dots \times 2_{X_*^{n_0}}\times \mathcal{U}_{4\epsilon_{n_0+1}}(A_{n_0+1}) \times \dots )\cap \mathcal{X}_* \\
V_2&=(2_{Y_*^1}\times 2_{Y_*^2}\times \dots \times 2_{Y_*^{n_0}}\times \mathcal{U}_{4\epsilon_{n_0+1}}(A_{n_0+1}) \times \dots )\cap \mathcal{X}_*.
\end{align*}
We argue by contradiction, suppose that the intersection of $V_1$ and $V_2$ is non-empty. Then there exists $\{ Z_*^n\}_{n\in \mathbb{N}}\in V_1\cap V_2$ such that $X_*^m,Y_*^m\subset Z_*^m$ for every $m\leq n_0$. We consider $x_{n_0}\in X_*^{n_0}$ and $y_{n_0}\in Y_*^{n_0}$. It follows that
\begin{align*}
16\epsilon_{n_0}<d(x,y)<d(x,x_{n_0})+d(x_{n_0},y_{n_0})+d(y_{n_0},y)< \\
<2\epsilon_{n_0}+ d(x_{n_0},y_{n_0})+2\epsilon_{n_0}=4\epsilon_{n_0}+d(x_{n_0},y_{n_0}).
\end{align*}
Therefore, $12\epsilon_{n_0}< d(x_{n_0},y_{n_0})$ but $\{Z_n\}_{n\in \mathbb{N}}\in \mathcal{X}_*$, so $diam(Z_{n_0})<4\epsilon_{n_0}$, which leads to a contradiction.
\end{proof}

\begin{thm}\label{thm:homeo}
$X$ is homeomorphic to $\mathcal{X}_*$.
\end{thm}
\begin{proof}
We have that $\phi:X\rightarrow \mathcal{X}_*$ is a continuous bijective map between a compact Hausdorff space and a Hausdorff space. Thus, $\phi$ is a homeomorphism.
\end{proof}

\begin{rem} We can also prove Theorem \ref{thm:homeo} without Proposition \ref{prop:hausdorff}. We have that $\varphi_{|\mathcal{X}_*}:\mathcal{X}_*\rightarrow X$ is a continuous and bijective function. On the other hand, $\phi:X\rightarrow \mathcal{X}_*$ is a continuous and bijective function that verifies $\phi\circ \varphi_{|\mathcal{X}_*}=id_{\mathcal{X}_*}$.
\end{rem}

\begin{thm}\label{thm:strongretract}
$\mathcal{X}_*$ is a strong deformation retract of $\mathcal{X}$.
\end{thm}
\begin{proof}
It is easy to check that $\varphi\circ \phi :X\rightarrow X$ is the identity map. We will check that $\phi \circ \varphi: \mathcal{X}\rightarrow \mathcal{X}$ is homotopic to the identity map $id_{\mathcal{X}}$. We consider $H:\mathcal{X}\times I \rightarrow \mathcal{X}$ given by 
\[   
H(\{C_n\}_{n\in \mathbb{N}},t) = 
     \begin{cases}
       \{C_n\}_{n\in \mathbb{N}} \quad  t\in [0,1) \\
       \phi(\varphi(\{C_n\}_{n\in \mathbb{N}})) \quad t=1,
     \end{cases}
\]
where $I$ denotes the unit interval. To study the continuity of $H$, it is only necessary to check the continuity at the point $(\{C_n\}_{n\in \mathbb{N}},1)\in \mathcal{X}\times I$. For every neighborhood $W$ of $\phi(\varphi(\{C_n\}_{n\in \mathbb{N}}))=\{X_*^n\}_{n\in \mathbb{N}}$, we can obtain a neighborhood $V$ of the form 
$$V =(2_{X_*^1}\times 2_{X_*^2}\times \dots \times 2_{X_*^{r}}\times \mathcal{U}_{4\epsilon_{r+1}}(A_{r+1}) \times \dots )\cap \mathcal{X} $$
such that $V\subset W$. By the continuity of $\phi \circ \varphi$ we know that there exists an open neighborhood $U$ of $\{C_n \}_{n\in \mathbb{N}}$ with $\phi (\varphi(U))\subset V$. We take $\{D_n\}_{n\in \mathbb{N}}\in U$ and we denote $\phi (\varphi(\{D_n\}_{n\in \mathbb{N}}))=\{Y_*^n\}_{n\in \mathbb{N}}$. By the continuity, $\{Y_*^n\}_{n\in \mathbb{N}}\in V$, so $X_*^n \subset Y_*^{n}$ for every $n\leq r$. By Proposition \ref{prop:minimal}, we also know $Y_*^{m}\subset D_m$ for every $m\in \mathbb{N}$. Concretely, $X_*^n \subset D_{n}$ for every $n\leq r$. Therefore, $H(\{D_n\}_{n\in \mathbb{N}},t)=\{D_n\}_{n\in \mathbb{N}}\in V $ when $t\in [0,1)$ and $H(\{D_n\}_{n\in \mathbb{N}},1)=\{Y_*^n\}_{n\in \mathbb{N}}\in V$ clearly. Thus, $U\times I$ satisfies that $H(U\times I)\subset V$.
\end{proof}

We update the example introduced in Section \ref{section:constructionFASO} with the theory developed in this section. Since the elements of $\mathcal{X}_*$ have a constructive description, they can be computed.
\begin{ex} We study the elements of the inverse limit of the FASO constructed in Example \ref{ex:FASoS1primeraparte} for $S^1$.  By construction, we have that $A_n\subset A_{n+1}$ for every $n\in \mathbb{N}$. We study two cases to get a description of $\mathcal{X}_*$. Firstly, we show one useful property.

\underline{\textbf{Assertion.}} If $a^n_k,a^n_{k+1}\in A_n$, where $k,k+1\in \{0,1,...,2^{3n-4}-1\}$, then $a^n_k $ and $a^n_{k+1}$ are in an arc of $S^1$ formed by two consecutive points $a^{n-1}_l,a^{n-1}_{l+1}\in A_{n-1}$ for some $l$ such that the length of the arc is $\frac{\pi}{2^{3(n-1)-5}}$.

\begin{proof}
We can determine $l$ solving the following equation:
$\frac{2 \pi k}{2^{3n-4}}= \frac{2 \pi l}{2^{3n-7}}$. Then, $l=\frac{k}{2^3}$. We have two possibilities:
\begin{itemize}
\item $l\in \mathbb{N}$, which implies that $a^n_k\in A_{n-1}$ and the result follows easily.
\item $l$ is not a natural number. We define $l^*$ as the integer part of $l$. We will check that $a^{n-1}_{l^*}, a^{n-1}_{l^*+1}\in A_{n-1}$ are the desired points. We know that $k=8a+r$ for some $a$ and $r<8$. Therefore,
$$\frac{2 \pi (k+1)}{2^{3n-4}}= \frac{2 \pi m}{2^{3n-7}}\Rightarrow m=\frac{k+1}{8}=a+\frac{r+1}{8}.$$
If $\frac{r+1}{8}$ is an integer, we are in the first case. We suppose that it is not an integer number. Then, $\frac{r+1}{8}<1$ and the integer part of $m$ is exactly $l^*$.
\end{itemize}
\end{proof}

If $x\in X$, then $\{ X_*^n\}_{n\in \mathbb{N}}$ can only be of two different forms.

\textbf{\underline{Case 1:}}  $x\in S^1$ satisfies that $x=a_s^n\in A_n$ for some $n\in \mathbb{N}$. Therefore, $a_s^n\in A_m$ for every $m$ satisfying $n\leq  m$. We want to describe $\phi (x)=\{X_*^n\}_{n \in \mathbb{N}}$. On the one hand, it is clear that $A_m(a^n_s)=a_s^n$ for every $m\geq n$. On the other hand, $q_{m,m+1}(A_{m+1}(a^n_s))=q_{m,m+1}(a^n_s)=\mathcal{B}(a_s^n,\epsilon_m)\cap A_m=a_s^n$ for every $m\geq n$ because $d(a_s^n,A_m\setminus \{a_s^n\})=\epsilon_m$. From here, we can deduce that $X_*^m=\{a_s^n\}$ for every $m\geq n$. We study $X_*^m$ with $m<n$. By the previous observation, $q_{m,t}(A_t(a_s^n))=q_{m,n}(a_s^n)$ if $t>n$. Clearly, $a_s^n$ is between two consecutive points $a_{k}^{n-1},a_{k+1}^{n-1}\in A_{n-1}$. Therefore, $q_{n-1,n}(a_s^n)=\{ a_{k}^{n-1},a_{k+1}^{n-1}\}$. By the previous assertion, $a_{k}^{n-1}$ and $a_{k+1}^{n-1}$ are between two consecutive points $a_k^{n-2},a_{k+1}^{n-2}\in A_{n-2}$ so $q_{n-2,n-1}(\{a_{k}^{n-1},a_{k+1}^{n-1}\})=\{a_{k}^{n-2},a_{k+1}^{n-2} \}$. Thus, we can deduce that $X_*^m=\{ a_{k}^{m},a_{k+1}^{m} \}$ for every $m<n$ and $m\neq 1$ because $X_*^1=\{e^{2\pi i} \}$. In Figure \ref{fig:caso1} we present a schematic draw of the above description.
$$\{X_*^n\}_{n\in \mathbb{N}}=(\{e^{2\pi i} \}, \{a_k^{1},a_{k+1}^{1} \},  \{a_k^{2},a_{k+1}^{2}\},\dots,  \{a_k^{n-1},a_{k+1}^{n-1} \}, \{a_s^n\}, \{a_s^n \}, \dots \} ) $$

\textbf{\underline{Case 2:}}  $x\in S^1$ such that $x\notin A_n$ for every $n\in \mathbb{N}$. We are interested in the following property: for every $n$ there exists $m>n$ such that $A_m(x)\notin A_{m-1}$. We argue by contradiction. Suppose $A_m(x)\in A_{n+1}$ for every $m>n$, which implies $x\in A_{n+1}$ and then a contradiction. By Proposition \ref{prop:estabiliza}, we know that for each $n$ there exists $n_*>n$ such that $\{X_*^n \}=q_{n,m}(A_m(x))$ for every $m>n_*$. On the other hand, there exists $s>n_*$ with $A_s(x)\notin A_{s-1}$ and $\{X_*^n \}=q_{n,s}(A_s(x))$. We get that $A_s(x)$ is at most two consecutive points $a_k^{s},a_{k+1}^{s}\in A_s$. Furthermore, by the previous property, $a_k^{s},a_{k+1}^{s}$ are between two consecutive points $a_k^{s-1},a_{k+1}^{s-1}\in A_{s-1}$. Hence, $q_{s-1,s}(\{ a_k^{s},a_{k+1}^{s}\})=\{a_k^{s-1},a_{k+1}^{s-1} \}$. Thus, we can deduce that $X_*^n=\{a_k^{n},a_{k+1}^{n}\}$ where $a_k^{n},a_{k+1}^{n}\in A_n$ and x lies in the arc formed by $a_k^{n}$ and $a_{k+1}^{n}$. In Figure \ref{fig:caso2} we present a schematic draw of the above description.

\begin{figure}[!htb]
   \begin{minipage}{0.48\textwidth}
   \centering
     \includegraphics[scale=0.8]{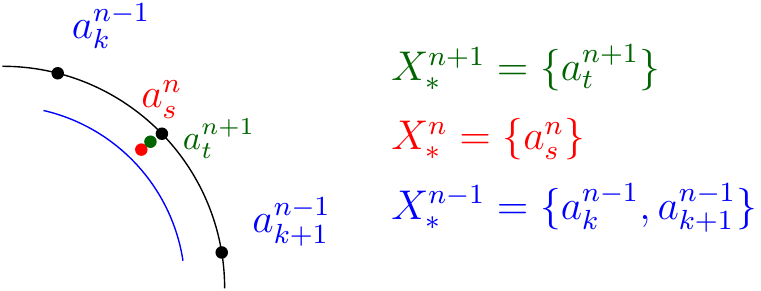}
\caption{Schematic situation when $x$ \newline is equal to $a_n\in A_n$ but $x\notin A_{n-1}$.}\label{fig:caso1}
   \end{minipage}\hfill \quad \quad\quad
   \begin{minipage}{0.48\textwidth}
        \centering
     \includegraphics[scale=0.8]{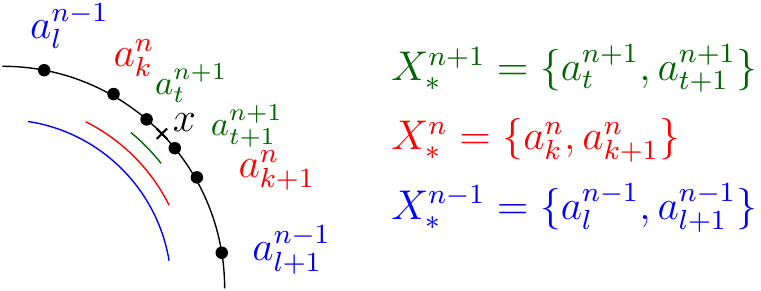}
\caption{Schematic situation when $x\notin A_n$ for every $n\in \mathbb{N}$.}\label{fig:caso2}
   \end{minipage}
\end{figure}
\end{ex}

\section{Uniqueness of the FASO constructed for a compact metric space and relations with other constructions}\label{sec:uniqueness}

Given a compact metric space $(X,d)$, a FASO $(\mathcal{U}_n,q_{n,n+1})$ for $X$ is not unique since it depends on the values of $\epsilon_n$ and the points $A_n$ that we choose. On the other hand, the main results obtained in Section \ref{sec:inverseLimitProperties} do not depend on the values and points chosen. A FASO $(\mathcal{U}_n,q_{n,n+1})$ for $X$ can be seen as an object of pro-$HTop$ because $(\mathcal{U}_n,q_{n,n+1})$ is an inverse sequence, where $HTop$ denotes the homotopical category of topological spaces. For a brief introduction  about pro-categories, see Appendix. For a complete exposition about pro-categories, see \cite{mardevsic1982shape}.

\begin{thm}\label{thm:unicidadFASO} Let $(X,d)$ be a compact metric space. If $(\mathcal{U}_{4\epsilon_n}(A_n),q_{n,n+1})$ is a FASO for $X$ and $(\mathcal{U}_{4\delta_n}(B_n),q_{n,n+1})$ is a different FASO for $X$, then $(\mathcal{U}_{4\epsilon_n}(A_n),q_{n,n+1})$ and $(\mathcal{U}_{4\delta_n}(B_n),q_{n,n+1})$ are isomorphic in pro-$HTop$.
\end{thm}
\begin{proof}
Firstly, we define a candidate to be an isomorphism. We consider $I:\mathbb{N}\rightarrow \mathbb{N}$ given by $I(n)=min\{ l\in \mathbb{N}|\delta_l<\frac{\epsilon_n}{16}\}$ and $I_n:\mathcal{U}_{4\delta_{I(n)}}(B_{I(n)})\rightarrow \mathcal{U}_{4\epsilon_n}(A_n)$ given by $I_n(C)=\bigcup_{x\in C}\mathcal{B}(x,\epsilon_n)\cap A_n$. We prove that $(I_n,I)$ is a morphism in pro-$HTop$. For simplicity, we omit some subscripts when there is no confusion.

We prove that $I_n$ is well-defined for every $n\in \mathbb{N}$. Suppose $C\in \mathcal{U}_{4\delta_{I(n)}}(B_{I(n)})$. If $x,y\in I_n(C)$, then there exist $a_x,a_y\in C$ such that $x\in I_n(a_x)$ and $y\in I_n(a_x)$. We get that $d(x,a_x),d(y,a_y)<\epsilon_n$. We also have that $diam(C)<4\delta_{I(n)}<\frac{\epsilon_n}{4}$, which implies that 
$$ d(x,y)<d(x,a_x)+d(a_x,y)+d(a_y,y)<\epsilon_n+\frac{\epsilon_n}{4}+\epsilon_n<4\epsilon_n.$$
Thus, $I_n$ is well-defined for every $n\in \mathbb{N}$. The continuity of $I_n$ follows trivially because $I_n$ clearly preserves the order. Now, we check that for every $m\geq n$, where $n,m\in \mathbb{N}$, the following diagram is commutative up to homotopy.
\[
  \begin{tikzcd}[row sep=large,column sep=huge]
\mathcal{U}_{4\delta_{I(n)}}(B_{I(n)}) \arrow{d}{I_n}  & \mathcal{U}_{4\delta_{I(m)}}(B_{I(m)}) \arrow{d}{I_m} \arrow[l,"q_{I(n),I(m)}"']  \\
\mathcal{U}_{4\epsilon_n}(A_n) &  \mathcal{U}_{4\epsilon_m}(A_m) \arrow{l}{q_{n,m}} 
    \end{tikzcd} 
\] 

Suppose $C\in \mathcal{U}_{4\delta_{I(m)}}(B_{I(m)})$. If $x\in I_n(q(C))$, then there exist $a_x\in B_{I(n)}, b_x\in C$ satisfying that $a_x\in q(b_x)$ and $x\in I_n(a_x)$. Therefore, $d(b_x,a_x)<2\delta_{I(n)}<\frac{\epsilon_n}{8}$ and $d(a_x,x)<\epsilon_n$. If $y\in q(I_m(C))$, then there exist $a_y\in A_m$, $b_y\in C$ such that $y\in q(a_y)$ and $a_y\in I_m(b_y)$. We get $d(y,a_y)<2\epsilon_n$ and $d(a_y,b_y)<\epsilon_m<\frac{\epsilon_n}{2}$. We know $d(b_y,b_x)<4\delta_{I(m)}<\frac{\epsilon_m}{4}<\frac{\epsilon_n}{8}$. Hence,
\begin{align*}
d(x,y)< &d(x,a_x)+d(a_x,b_x)+d(b_x,b_y)+d(b_y,a_y)+d(a_y,y)\\
<& \epsilon_n+\frac{\epsilon_n}{8}+\frac{\epsilon_n}{8}+\frac{\epsilon_n}{2}+2\epsilon_n<4\epsilon_n
\end{align*}
We can conclude that $diam(I_n(q(C))\cup q(I_m(C)))<4\epsilon_n$ for every $C\in \mathcal{U}_{4\delta_{I(m)}}(B_{I(m)})$. Thus, $I_n\circ q\cup q \circ I_m:\mathcal{U}_{4\delta_{I(m)}}(B_{I(m)})\rightarrow \mathcal{U}_{4\epsilon_n}(A_n)$  is well-defined, continuous and satisfies $I_n(q(C))$, $q(I_m(C))\subset I_n(q(C)) \cup q(I_m(C))$ for every $C\in \mathcal{U}_{4\delta_{I(m)}}(B_{I(m)})$. From here, we get that the previous diagram is commutative up to homotopy.

We consider $T:\mathbb{N}\rightarrow \mathbb{N}$ given by $T(n)=min\{ l\in \mathbb{N}|\epsilon_l<\frac{\delta_n}{16}\}$ and $T_n:\mathcal{U}_{4\epsilon_{T(n)}}(A_{T(n)})\rightarrow \mathcal{U}_{4\delta_n}(B_n)$ given by $T_n(C)=\bigcup_{x\in C}\mathcal{B}(x,\delta_n)\cap B_n$. $(T_n,T)$ is a well-defined morphism in pro-$HTop$. To prove the last assertion we only need to repeat the previous arguments.

Now, we prove that $(T_n,T)\circ(I_n, I)$ is homotopic to $(id_\delta,id):(\mathcal{U}_{4\delta_n(B_n)},q_{n,n+1})\rightarrow (\mathcal{U}_{4\delta_n(B_n)},q_{n,n+1})$, where $(id_\delta,id)$ denotes the identity morphism. 

We consider $m\in \mathbb{N}$ satisfying that $m>I(T(n))$. Hence, we need to show that the following diagram is commutative up to homotopy, where $id$ denotes the identity map.
\[
  \begin{tikzcd}[row sep=large,column sep=large]
\mathcal{U}_{4\epsilon_{T(n)}}(A_{T(n)}) \arrow[drr,"T_n"'] & \mathcal{U}_{4\delta_{I(T(n))}}(B_{I(T(n))}) \arrow[l,"I_{T(n)}"'] & \mathcal{U}_{4\delta_m}(B_m) \arrow[l,"q_{I(T(n)),m}"'] \arrow{r}{q_{n,m}}&\mathcal{U}_{4\delta_n}(B_n) \arrow{dl}{id}  \\
& & \mathcal{U}_{4\delta_n}(B_n) &
  \end{tikzcd}
\]
Suppose that $C\in \mathcal{U}_{4\delta_m}(B_m)$. If $x\in q(C)$, then there exists $c_x\in C$ such that $x\in q(c_x)$ and $d(x,c_x)<2\delta_n$. If $y\in T(I(q(C)))$, then there exist $a_y\in A_{T(n)}$, $b_y\in B_{I(T(n))}$ and $c_y\in C$ such that $y\in T(a_y)$, $a_y\in I(b_y)$, $b_y\in q(c_y)$. Therefore, $d(y,a_y)<\delta_n$, $d(a_y,b_y)<\epsilon_{I(n)}<\frac{\delta_n}{16}$, $d(b_y,c_y)<2\delta_{I(T(n))}<\frac{\epsilon_{I(n)}}{8}<\frac{\delta_n}{128}$. On the other hand, $diam(C)<4\delta_{m}$ and $d(c_x,c_y)<4\delta_m<2\delta_{I(T(n))}<\frac{\delta_n}{128}$, which implies
\begin{align*}
d(x,y)<&d(x,c_x)+d(c_x,c_y)+d(c_y,b_y)+d(b_y,a_y)+d(a_y,y)\\
<& 2\delta_n+\frac{\delta_n}{128}+\frac{\delta_n}{128}+\frac{\delta_n}{16}+\delta_n<4\delta_n.
\end{align*}
We get $diam(T(I(q(C)))\cup q(C))<4\delta_n$ for every $C\in \mathcal{U}_{4\delta_m}(B_m)$. Thus, $h=T\circ I\circ q \cup q:\mathcal{U}_{4\delta_m}(B_m)\rightarrow \mathcal{U}_{4\delta_n}(B_n)$ is a well-defined and continuous map. Furthermore, $I(I(q(C))),q(C)\subset h(C)$ for every $C\in \mathcal{U}_{4\delta_m}(B_m)$, which implies that the diagram is commutative up to homotopy.

Repeating the same arguments, it can be deduced that $(I_n,I)\circ(T_n, T)$ is homotopic to the identity morphism $(id_\epsilon,id):(\mathcal{U}_{4\epsilon_n}(A_n),q_{n,n+1})\rightarrow (\mathcal{U}_{4\epsilon_n}(A_n),q_{n,n+1})$.
\end{proof}

Now, we see the relations between the inverse sequences considered throughout this manuscript as objects of pro-$HTop$. Firstly, given a compact metric space $(X,d)$, it can be considered the FAS or Main Construction $(\mathcal{U}_{2\epsilon_n}(A_n), p_{n,n+1})$ for $X$, \cite{moron2008connectedness,mondejar2020reconstruction}. We can consider the opposite order for each term of this inverse sequence. Then, we can compare this inverse sequence with a FASO $(\mathcal{U}_{4\delta_n}(B_n), q_{n,n+1})$ for $X$.

\begin{thm}\label{thm:FASOequalFAS} Given a compact metric space $(X,d)$. If $(\mathcal{U}_{2\epsilon_n}(A_n), p_{n,n+1})$ is a FAS for $X$, where each term is considered with the opposite order, then every FASO $(\mathcal{U}_{4\delta_n}(B_n),q_{n,n+1})$ for $X$ is isomorphic to $(\mathcal{U}_{2\epsilon_n}(A_n), p_{n,n+1})$ in pro-$HTop$.
\end{thm}
\begin{proof}
It is easy to show that $(\mathcal{U}_{4\epsilon_n}(A_n),q_{n,n+1})$ is a FASO for $X$. By Theorem \ref{thm:unicidadFASO}, we have that $(\mathcal{U}_{4\epsilon_n}(A_n),q_{n,n+1})$ is isomorphic to $(\mathcal{U}_{4\delta_n}(B_n),q_{n,n+1})$. Therefore, it only remains to show that $(\mathcal{U}_{4\epsilon_n}(A_n),q_{n,n+1})$ is isomorphic to $(\mathcal{U}_{2\epsilon_n}(A_n), p_{n,n+1})$.

We have a natural inclusion $i_n:\mathcal{U}_{2\epsilon}(A_n)\rightarrow  \mathcal{U}_{4\epsilon}(A_n)$ for every $n\in \mathbb{N}$. The following diagram is commutative up to homotopy due to the fact that $p(i(C))\subseteq i(q(C))$ for every $C\in \mathcal{U}_{2\epsilon_{n}}(A_n)$.

\[
  \begin{tikzcd}[row sep=large,column sep=huge]
\mathcal{U}_{2\epsilon_n}(A_n) \arrow{d}{i} &\mathcal{U}_{2\epsilon_{n+1}}(A_{n+1})  \arrow{d}{i}  \arrow{l}{p_{n,n+1}}  \\
\mathcal{U}_{4\epsilon_n} (A_n)    & \mathcal{U}_{4\epsilon_{n+1}}(A_{n+1}) \arrow{l}{q_{n,n+1}}
  \end{tikzcd}
\]
Then, we have a level morphism between the two inverse sequences considered. For every $m>n$, we define $g_n:\mathcal{U}_{4\epsilon_m}(A_m)\rightarrow \mathcal{U}_{2\epsilon_n}(A_n)$  given by 
$$g_n(C)=\overline{\mathcal{B}}(C,\gamma_n)\cap A_n=\bigcup_{x\in C}\overline{\mathcal{B}}(x,\gamma_n)\cap A_n,$$ 
where $\overline{\mathcal{B}}(x,\gamma_n)$ denotes the closed ball of radius $\gamma_n$ and $\gamma_n=\sup \{ d(x,A_n)|x\in X\}$. We prove that $g_n$ is well-defined. If $x,y\in g_n(C)$, then there exist $c_x,c_y\in C$ with $x\in g_n(c_x),y\in g_n(c_y)$, so $d(x,c_x),d(y,c_y)\leq \gamma_n$. We also know that $diam(C)<4\epsilon_m$, which implies $d(c_x,c_y)<4\epsilon_m<2\epsilon_n-2\gamma_n$. Therefore,
$$d(x,y)<d(x,c_x)+d(c_x,c_y)+d(c_y,y)<\gamma_n+2\epsilon_n-2\gamma_n+\gamma_n=2\epsilon_n $$
The continuity of $g_n$ for every $n\in \mathbb{N}$ follows trivially. Since $g_n(i(C)) \subset p(C)$ and $q(C) \subset i(g_n(C))$, we get that $g_n\circ i$ is homotopic to $p$ and $q$ is homotopic to $i\circ g_n$. By Morita's lemma (Theorem \ref{thm:Moritalema}), see \cite{morita1974hurewicz} or \cite[Chapter 2, Theorem 5]{mardevsic1982shape}, we get the desired result.
\end{proof}

Given a polyhedron $K$, we can obtain an inverse sequence of finite topological spaces using the theory developed previously or we can apply the construction made in \cite{clader2009inverse}. Let $(X^n,h_{n,n+1})$ denote the construction obtained in \cite{clader2009inverse}, where $X^n$ denotes the $n$-th finite barycentric subdivision of $\mathcal{X}(K)$ with the opposite order and $h_{n,n+1}$ is the natural map which sends the chain $x_1<...<x_n$ of $X^{n}$ to $x_1\in X^n$. We introduce a bit of notation. Let $\mathcal{X}(K)$ denote the face poset of $K$. Given a poset $X$. Let $\mathcal{K}(X)$ denote the order complex of $X$. The finite barycentric subdivision of $X$ is given by $\mathcal{X}(\mathcal{K}(X))$. For a complete exposition about this and other properties, see \cite{may1966finite}.

\begin{thm}\label{thm:claderComparacionFASO} Let $K$ be a finite polyhedron. If $(\mathcal{U}_{4\epsilon_n}(A_n),q_{n,n+1})$ is a FASO for $K$, then there is a natural morphism $(i_n,id):(X^n,h_{n,n+1})\rightarrow(\mathcal{U}_{4\epsilon_n},q_{n,n+1})$ in pro-$HTop$.
\end{thm}
%
\begin{proof} 
We construct a new FASO $(\mathcal{U}_{4\sigma_n}(B_n),q_{n,n+1})$ for $K$ and a morphism $(i_n,id):(X^n,h_{n,n+1})\rightarrow (\mathcal{U}_{4\sigma_n}(B_n),q_{n,n+1})$. Then, by Theorem \ref{thm:unicidadFASO}, we can conclude.  

\begin{figure}[h]
\centering
\includegraphics[scale=0.9]{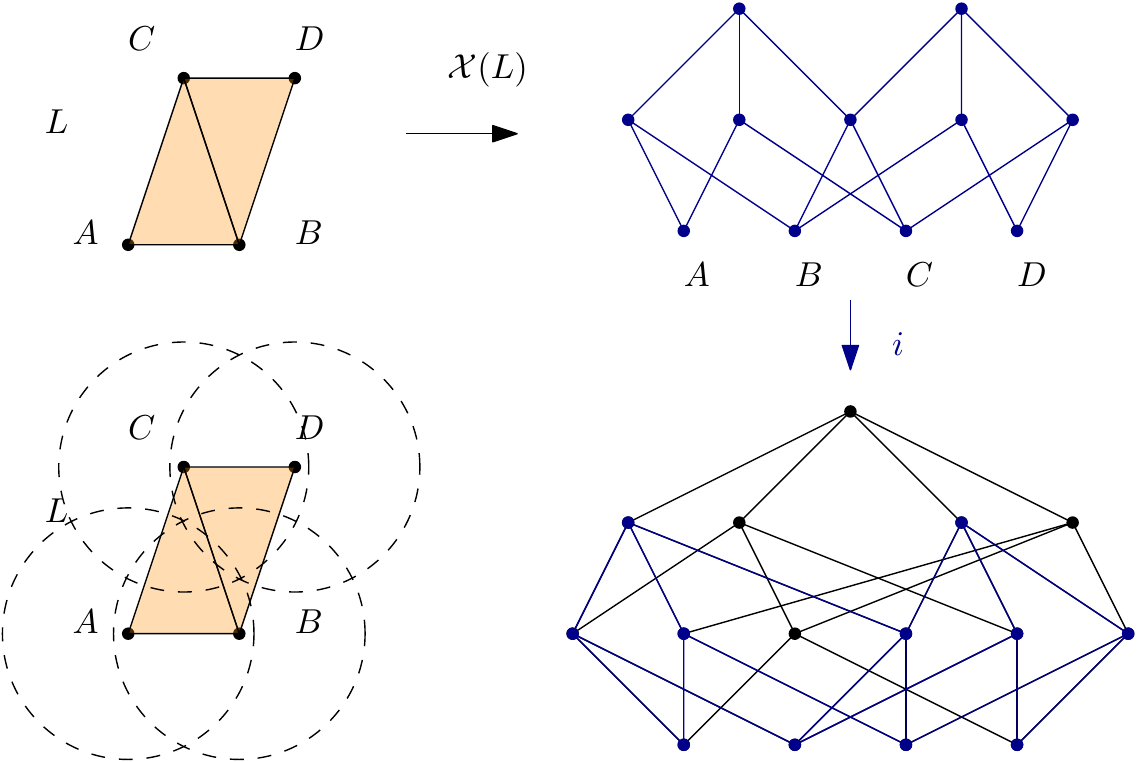}
\caption{Representation of the simplicial complex $L$ and Hasse diagrams of $\mathcal{X}(L)$ and $\mathcal{U}_{4\epsilon_n}(\{A,B,C,D \})$.}\label{fig:claderRelacion}
\end{figure}

We start taking $\sigma_1>diam(K)$, $B_1$ is the set of vertices of $K$ and $X^1=\mathcal{X}(K)$. We consider $\delta_1=\sup\{d(x,B_1)|x\in K\}$ and $\sigma_2<\frac{\sigma_1-\delta_1}{2}$. Applying barycentric subdivisions of $K$, we obtain that there exists $t_2\in \mathbb{N}$ such that the set of vertices of $K^{t_2}$ is a $\sigma_2$-approximation $B_2$ of $K$, $X^2=\mathcal{X}(K^{t_2})$ is a subposet of $\mathcal{U}_{4\sigma_2}(B_2)$ and $ q_{2,1}\circ i \geq i \circ h_{t_2,t_1}$, where $i$ denotes the inclusion map with abuse of notation.  We can deduce the last assertion using that the diameters of the simplices obtained after a barycentric subdivision are smaller than the original ones, see \cite{hatcher2000algebraic} or \cite{spanier1981algebraic}. In Figure \ref{fig:claderRelacion}, there is an example of the situation described above.

Arguing inductively, we can obtain a FASO $(\mathcal{U}_{4\sigma_n}(B_n),q_{n,n+1})$ for $K$ such that $X^{t_n}$ is a subposet of $\mathcal{U}_{4\sigma_n}(B_n)$. The set $\{t_n\}_{n\in \mathbb{N}}$ is a cofinal subset of $\mathbb{N}$. Then $(X^n,h_{n,n+1})$ is isomorphic to $(X^{t_n},h_{n,n+1},\{t_n\}_{n\in \mathbb{N}})$ in pro-$HTop$. We have that the following diagram commutes up to homotopy.
\[
  \begin{tikzcd}[row sep=large,column sep=huge]
\mathcal{U}_{4\sigma_n}(B_n) &\mathcal{U}_{4\sigma_{n+1}}(B_{n+1})    \arrow{l}{q_{n,n+1}}  \\
X^{t_n} \arrow{u}{i}    & X^{t_{n+1}}\arrow{u}{i} \arrow{l}{h_{t_n,t_{n+1}}}
  \end{tikzcd}
\]
\end{proof}

\begin{rem} In general, the continuous inclusion $i:X^{t_n}\rightarrow \mathcal{U}_{4\sigma_n}(B_n)$ considered in the proof of Theorem \ref{thm:claderComparacionFASO} is not a homotopy equivalence. 
\end{rem}

Given a compact metric space $X$, it is proved in \cite{mondejar2015hyperspaces} the following: if the functor $\mathcal{K}$ (see \cite{mccord1966singular}) is applied to the Main Construction, then it is obtained a $HPol$-expansion of $X$ (see \cite{mardevsic1982shape}). If $X$ is a finite $T_0$ topological space where $X$ denotes the poset with the natural order and $X^o$ denotes the poset with the opposite order, then $\mathcal{K}(X)=\mathcal{K}(X^o)$. From this, we deduce that if $(\mathcal{U}_{4\epsilon_n}(A_n),q_{n,n+1})$ is a FASO for $X$ and the functor $\mathcal{K}$ is applied to $(\mathcal{U}_{4\epsilon_n}(A_n),q_{n,n+1})$, then we get a $HPol$-expansion of $X$. This means that we can use the FASO for a compact metric space to reconstruct algebraic invariants of $X$.
\begin{rem}\label{rem:Apaño2}
Given a compact metric space $(X,d)$ and a FASO $(\mathcal{U}_{4\epsilon_n}(A_n),q_{n,n+1})$ for $X$. We can consider a decreasing sequence of positive values $\{\tau_n \}_{n\in\mathbb{N}}$ in the hypothesis of Remark \ref{rem:Apaño}, which is clearly less restrictive. For every $n\in \mathbb{N}$ we consider a $\tau_n$-approximation $B_n$ for $X$. We get an inverse sequence $(\mathcal{U}_{4\tau_n}(B_n),q_{n,n+1})$. Repeating the same arguments used in the proof of Theorem \ref{thm:unicidadFASO}, we can obtain that $(\mathcal{U}_{4\epsilon_n}(A_n),q_{n,n+1})$ is isomorphic to $(\mathcal{U}_{4\tau_n}(B_n),q_{n,n+1})$ in pro-$HTop$. This result is important for computational reasons because we can ignore a hard hypothesis to check, that is, it is not necessary to compute $\gamma_n=sup\{d(x,A_n)|x\in X \}$. But removing the previous hypothesis, we cannot expect to get Theorem \ref{thm:homeo} and Theorem \ref{thm:strongretract} for the new inverse sequence. This is due to the fact that  $(\mathcal{U}_{4\epsilon_n}(A_n),q_{n,n+1})$ and  $(\mathcal{U}_{4\tau_n}(B_n),q_{n,n+1})$ are just isomorphic in pro-$HTop$ and not in pro-$Top$, the pro-category of topological spaces. On the other hand, since both inverse sequences are isomorphic in pro-$HTop$ we get that $(\mathcal{U}_{4\tau_n}(B_n),q_{n,n+1})$ serves to approximate algebraic invariants of $X$ like the homology groups. If we apply the homological functor $H_*$ to $(\mathcal{U}_{4\epsilon_n}(A_n),q_{n,n+1})$ and $(\mathcal{U}_{4\tau_n}(B_n),q_{n,n+1})$, then we get two inverse sequences of groups that are isomorphic in the pro-category of groups, which means that the inverse limits of the inverse sequences of groups are the same. Thus, the new inverse sequence can be used in order to approximate or get the \v{C}ech homology groups of $X$. Note that if $X$ is a $CW$-complex, then the singular homology of $X$ coincides with the \v{C}ech homology of $X$.
\end{rem}

As an immediate consequence of this result and the previous remark we get the following result.
\begin{prop}[Reconstruction of homology groups]\label{prop:ReconstructionHomology} Let $(X,d)$ be a compact metric space and let $\{ \tau_n\}_{n\in \mathbb{N}}$ be a sequence of positive real values satisfying that $\tau_{n+1}<\frac{\tau_n}{2}$ for every $n\in \mathbb{N}$. Then for every inverse sequence $(\mathcal{U}_{\tau_n}(B_n),q_{n,n+1})$, where $B_n$ is a $\tau_n$-approximation, the inverse limit of $(H_k(\mathcal{U}_{\tau_n}(B_n)),H_k(q_{n,n+1}))$ is isomorphic to the $k$-th \v{C}ech homology group of $X$.
\end{prop}

\section{Computational aspects of a FASO for a compact metric space and implementation}\label{section:algoritmosComparativa}
We recall the notion of Vietoris-Rips complex, \cite{vietoris1927uber}. Given a finite set of point $S\subseteq \mathbb{R}^n$ for some $n\in \mathbb{N}$ and a positive real value $\epsilon$. The Vietoris-Rips complex is a simplicial complex given as follows $V_\epsilon(S)=\{ \sigma\subseteq S|d(u,v)\leq \epsilon, \ \forall u\neq v\in \sigma \}$. In Figure \ref{fig:VietorisRipsComplex} we have an example of a Vietoris-Rips complex. For a complete introduction, see for example \cite{edelsbrunner2010computational}.

\begin{figure}[h]
\centering
\includegraphics[scale=1.3]{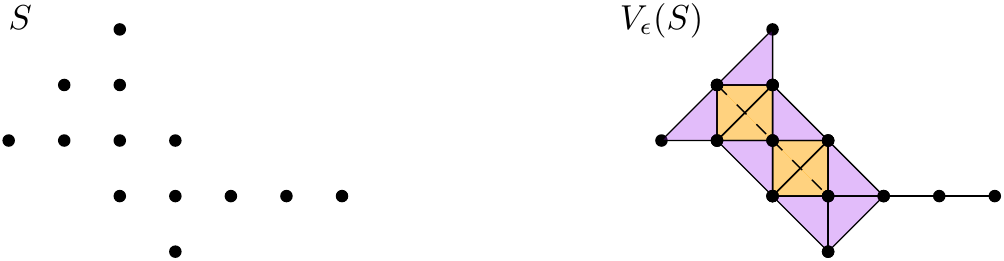}
\caption{A set of points $S\subset \mathbb{R}^2$ and Vietoris-Rips complex of $S$ for some positive value $\epsilon$.}\label{fig:VietorisRipsComplex}
\end{figure}

Let us assume that $X$ is a compact metric space embedded in $\mathbb{R}^n$ for some $n\in \mathbb{N}$. Given an $\epsilon_n$-approximation $A_n$, the problem of finding $\mathcal{U}_{4\epsilon_n}(A_n)$ is equivalent to construct the Vietoris-Rips complex $V_{4\epsilon_n}(A_n)$. This is due to the fact that $\mathcal{U}_{4\epsilon_n}(A_n)$ is the face poset of $V_{4\epsilon_n}(A_n)$, that is, $\mathcal{X}(V_{4\epsilon_n}(A_n))=\mathcal{U}_{4\epsilon_n}(A_n)$. Furthermore, we get that $V_{4\epsilon_n}(A_n)$ has the same weak homotopy type of $\mathcal{U}_{4\epsilon_n}(A_n)$ by \cite{mccord1966singular}. In \cite{zomorodian2010fast}, it is obtained an algorithm to get Vietoris-Rips complexes. Furthermore, it is also compared with other algorithms in terms of computational time. The experiments show that the algorithm introduced in \cite{zomorodian2010fast} is the fastest. It is also important to observe that this algorithm has two phases. In the first one, it is constructed the 1-skeleton of the Vietoris-Rips complex. Once it is obtained the 1-skeleton, the problem to obtain the entire simplicial complex is a combinatorial one and it is related to the computation of clique complexes. A clique is a set of vertices in a graph that induces a complete subgraph. A clique complex has the maximal cliques of a graph as its maximal simplices.

If we have computed $\mathcal{U}_{4\epsilon_n}(A_n)$ and $\mathcal{U}_{4\epsilon_{n+1}}(A_{n+1})$, then we need to get the map $q_{n,n+1}$. This problem is equivalent to a variation of a classical one in computational geometry, the $\epsilon$-nearest neighborhood problem. Namely, given a set of points $S\subset \mathbb{R}^n$ for some $n\in \mathbb{N}$, a point $q\in \mathbb{R}^n$ and a positive value $\epsilon$, find the set $N=\{x\in S|d(x,y)<\epsilon \}$.  This problem has been treated largely in the literature and has several approaches, see for instance \cite{Friedman1977algorithm} or \cite{arya1998optimal}. Using one of the previous algorithms it can be obtained $q_{n,n+1}$ over $A_{n+1}$. For the rest of the points in $\mathcal{U}_{4\epsilon_{n+1}}(A_{n+1})$, the description of $q_{n,n+1}$ is purely combinatorial since it is the union of the images of the points obtained before.

\begin{figure}[h]
\centering
\includegraphics[scale=0.9]{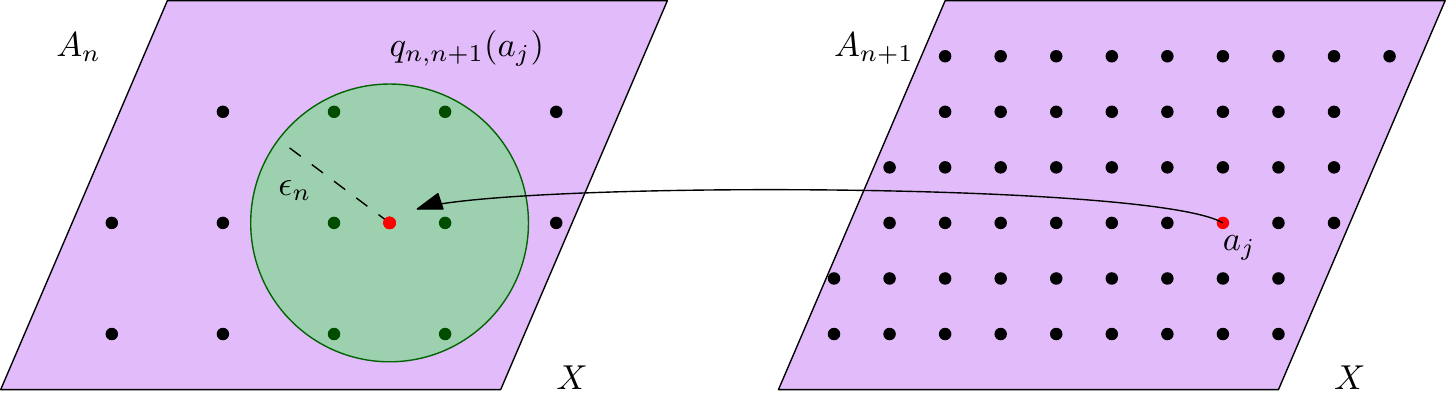}
\caption{Schematic description of $q_{n,n+1}(a_j)$, where $a_j\in A_{n+1}$.}

\end{figure}

Thus, if $X\subset \mathbb{R}^m$ for some $m\in \mathbb{N}$, then we have the following steps:
\begin{enumerate}
\item Find a sequence of positive values $\{\epsilon_n\}_{n\in \mathbb{N}}$ and an $\epsilon_n$-approximation $A_n$ for every $n\in \mathbb{N}$ satisfying the conditions required in the construction of a FASO for $X$.
\item Compute $\mathcal{U}_{4\epsilon_n}(A_n)$ for every $n\in \mathbb{N}$.
\item Obtain $q_{n,n+1}:\mathcal{U}_{4\epsilon_{n+1}}(A_{n+1})\rightarrow \mathcal{U}_{4\epsilon_n}(A_n)$ for every $n\in \mathbb{N}$.
\end{enumerate}
One possible approach to get the first step is to consider a sequence of grids. If we are only interested in algebraic invariants for $X$ such as the homology groups, then we can apply Remark \ref{rem:Apaño}, Remark \ref{rem:Apaño2} and Proposition \ref{prop:ReconstructionHomology} in order to simplify the computations. We have seen that the second step is equivalent to the construction of a Vietoris-Rips complex and the third step is equivalent to a classical problem in computational geometry. Then, combining classical algorithms that solve these problems we can obtain a FASO for a compact metric space embedded in $\mathbb{R}^m$. 

The method propose throughout this manuscript can also be used for data analysis. Instead of having a compact metric space embedded in $\mathbb{R}^n$ we could have just a set of points $S$ in $\mathbb{R}^n$. Using the metric inherit from $\mathbb{R}^n$, we can speak about $\epsilon$-approximations for the data set $S$. In this case, higher values of the inverse sequence tend to recover the data set as a disjoint union of points, that is, there exists $m$ such that for every $n\geq m$ we get $\mathcal{U}_{4\epsilon_n}(A_n)=\{ S\}$.

We present an easy example, where we assume that we cannot get in a deterministic way the points of the approximation, that is, we can obtain data from experiments or we only know that the data follow a distribution.
\begin{ex}\label{ex:ejemploCuadrados} We consider the space given by two squares in $\mathbb{R}^2$ as follows 
$$D=\{ 0\}\times [0,2]\cup \{ 1\}\times [0,2]\cup [0,1]\times \{0\}\cup [0,1]\times \{ 1\}\cup [0,1]\times \{ 2\}.$$ We can consider $\epsilon_n=\frac{1}{2^{2(n-1)}}$ by Remark \ref{rem:Apaño} and Remark \ref{rem:Apaño2}. Suppose that the data that can be taken to approximate $D$ follow a continuous uniform distribution. In Figure \ref{fig:observations} we have the data obtained for the first observations, concretely, $A_1$, $A_2$, $A_3$ and $A_4$.
\begin{figure}[h!]
\centering
\begin{tabular}{rc}
 \includegraphics[scale=0.58]{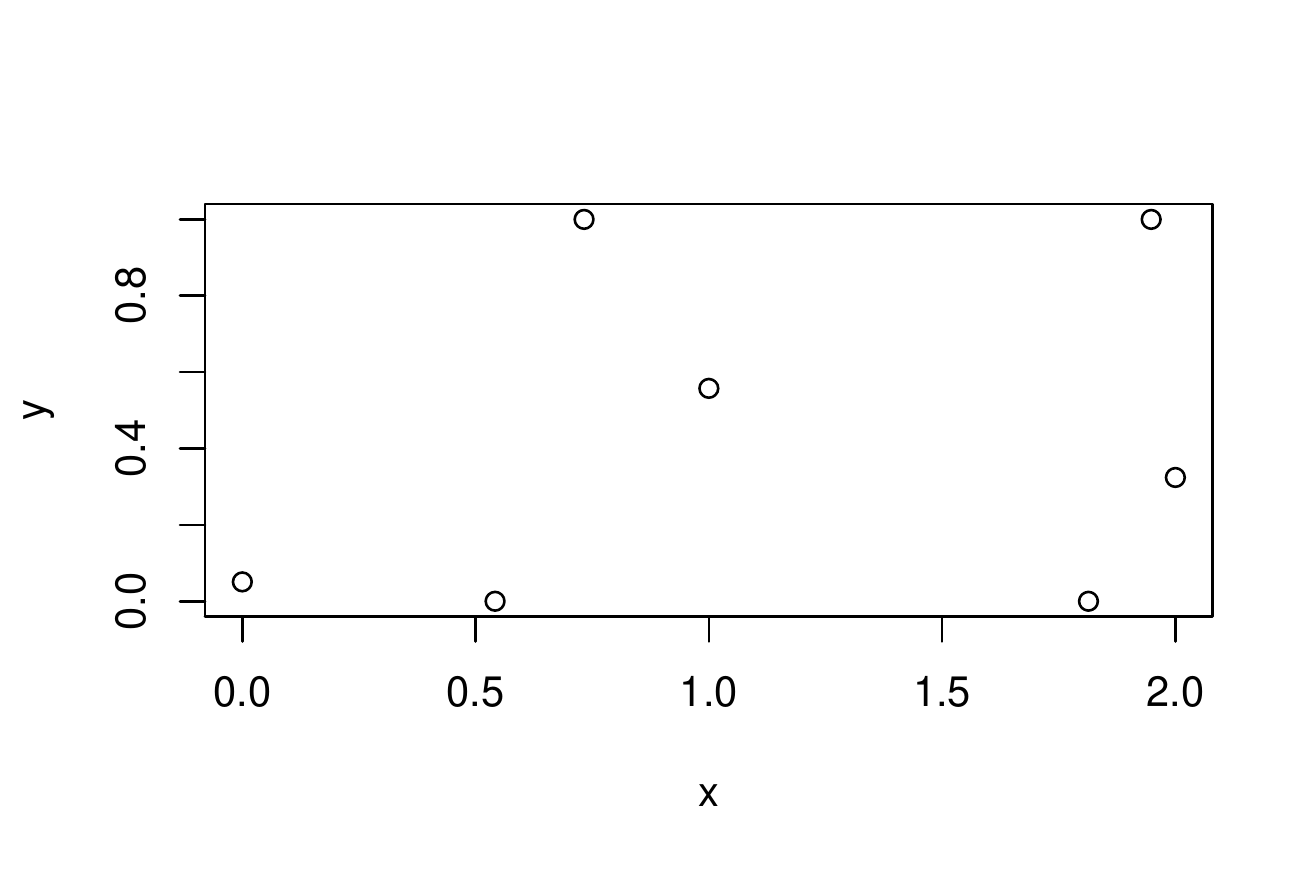}&
\includegraphics[scale=0.58]{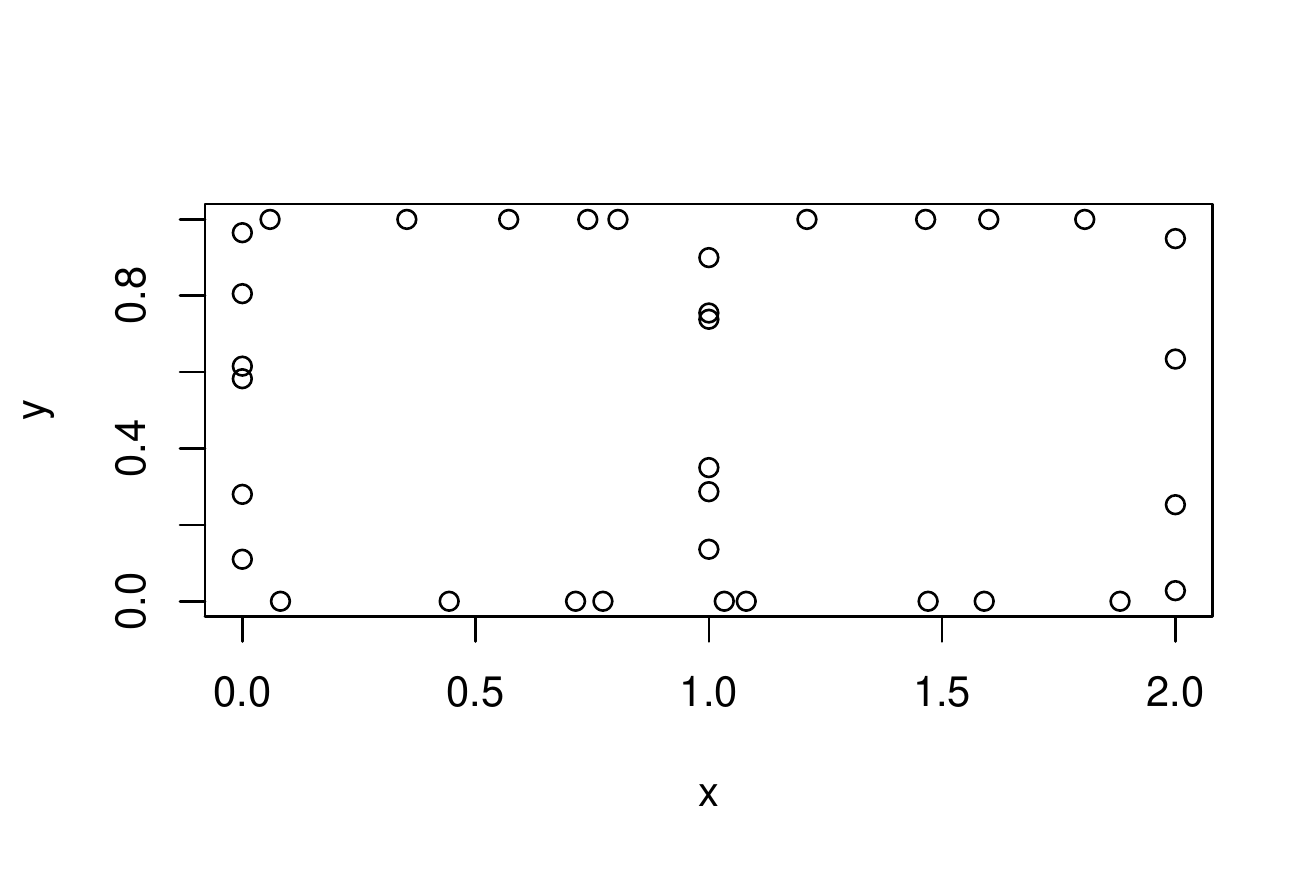}\\
 \includegraphics[scale=0.58]{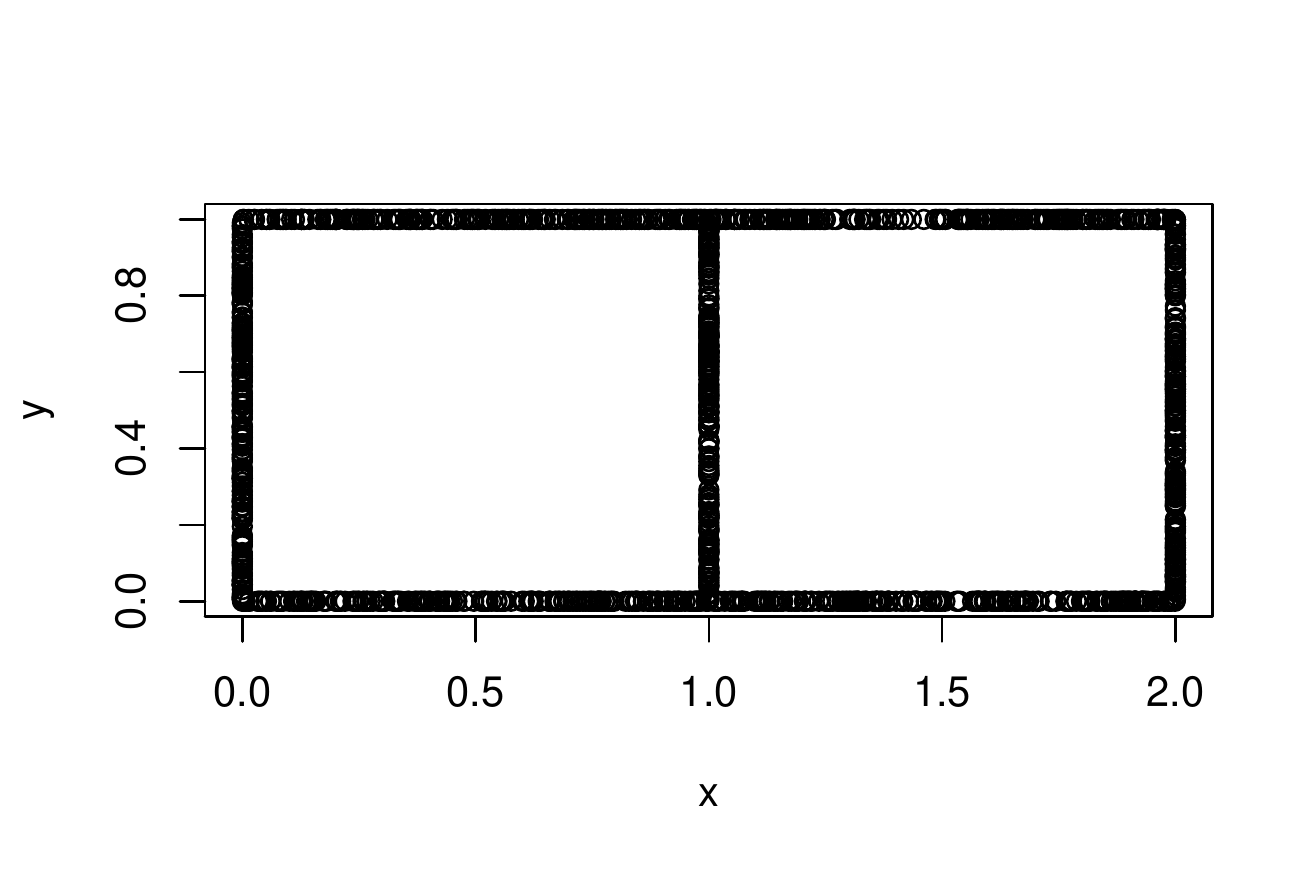}&
 \includegraphics[scale=0.58]{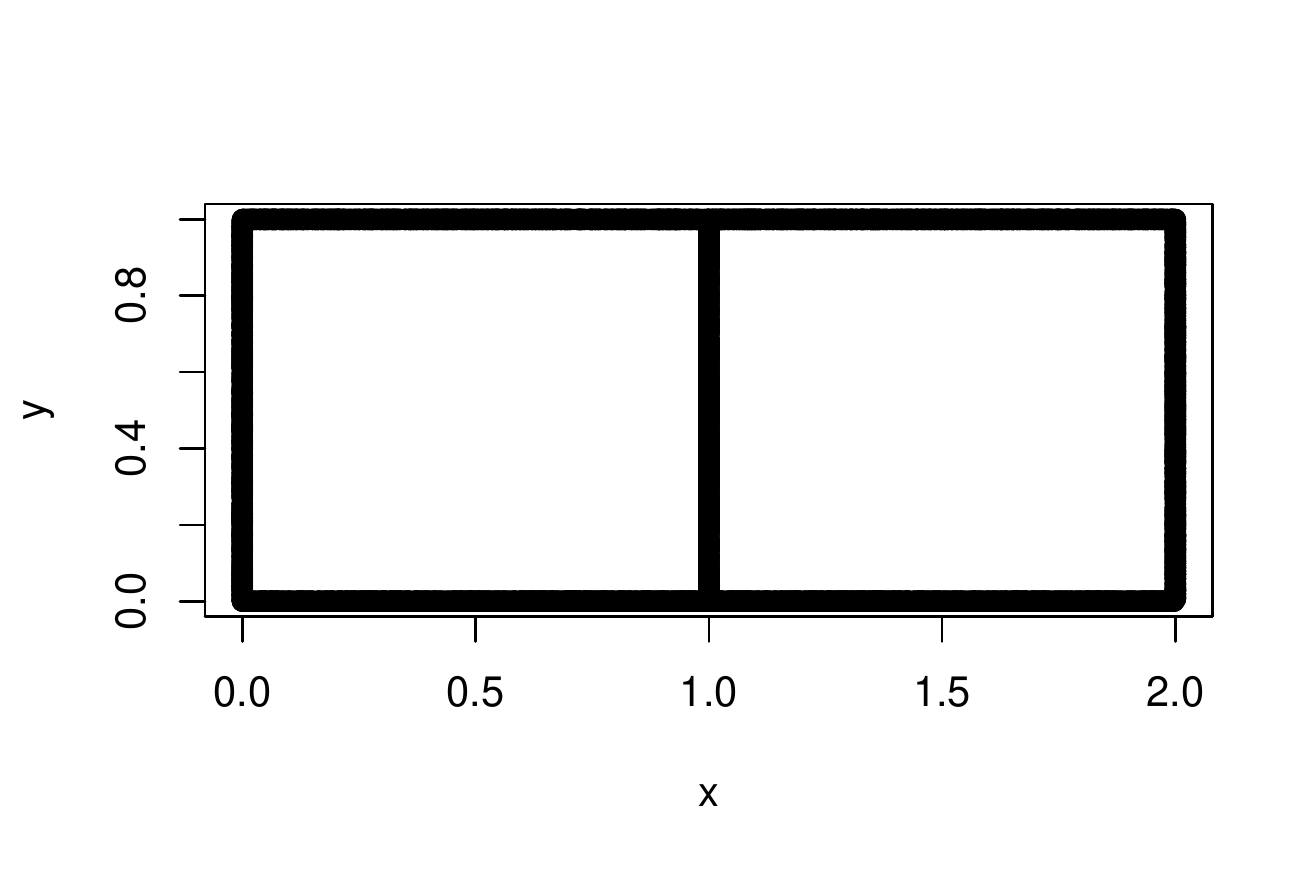}
\end{tabular}
\caption{$A_1$, $A_2$, $A_3$ and $A_4$.}\label{fig:observations}
\end{figure}
We compute the dimension of the homology groups of $\mathcal{U}_{4\epsilon_i}(A_i)$ for $i=1,..,5$, see Table \ref{table:dimensionesHomologia}.

\begin{table}[]\caption{Dimension of homology groups.}\label{table:dimensionesHomologia}
\centering
\begin{tabular}{c|ccccc}
      & $\mathcal{U}_{4\epsilon_1}(A_1)$ & $\mathcal{U}_{4\epsilon_2}(A_2)$ & $\mathcal{U}_{4\epsilon_3}(A_3)$ & $\mathcal{U}_{4\epsilon_4}(A_4)$ & $\mathcal{U}_{4\epsilon_5}(A_5)$ \\ \hline
$H_0$ & $1$                              & $1$                              & $1$                              & $1$                              & $1$                              \\
$H_1$ & $0$                              & $2$                              & $2$                              & $2$                              & $2$                              \\
$H_2$ & $0$                              & $0$                              & $0$                              & $0$                              & $0$                             
\end{tabular}
\end{table}

It seems that with a few steps we get a good representation of $D$ at least from a homological viewpoint. On the other hand, we know that different choices of the sequence $\{\epsilon_n\}_{n\in \mathbb{N}}$ or observations lead to the same results due to the results obtained in Section \ref{sec:uniqueness}, that are somehow telling that this method of approximation is robust.
\end{ex}

In Example \ref{ex:ejemploCuadrados}, we can observe that the dimension of homology groups stabilize very soon. This can happen due to the good properties of $D$, particularly, we have that $D$ is a connected compact $CW$-complex, which means that $D$ has good local properties. We provide an example for which this behavior does not happen. 
\begin{ex}\label{ex:ejemploCantor} We consider the Cantor set $C\subset[0,1]$. The exact construction is as follows. From the closed interval $E_1=[0,1]$, first remove the open interval $(\frac{1}{3},\frac{2}{3})$ leaving $E_2=[0,\frac{1}{3}]\cup [\frac{2}{3},1]$.  From $E_2$, delete the open intervals $(\frac{1}{9},\frac{2}{9})$ and $(\frac{7}{9},\frac{8}{9})$. $E_3$ is the remaining $4$ closed intervals. From $E_3$, remove middle thirds as before obtaining $E_4$. Hence, $C=\cap_{i=1}^\infty E_i$.  For a complete introduction and properties about the Cantor set $C$, see \cite{steen1978counterexamples}.

Again, it can be taken $\epsilon_n=\frac{1}{2^{2(n-1)}}$ by Remark \ref{rem:Apaño} and Remark \ref{rem:Apaño2}. We will use as approximations the endpoints of the closed intervals that remain after removing open intervals in every step, for example, $A_1=\{0,\frac{1}{3},\frac{2}{3},1 \}$. $A_n$ is given by the endpoints of $E_{n+1}$. It is not difficult to show that the endpoints of the remaining closed intervals belong to $C$ and form a $\epsilon_n$-approximation for every $n\in \mathbb{N}$. We present in Figure \ref{fig:observationsCantor} the first approximations.

\begin{figure}[h!]
\centering
\begin{tabular}{rc}
 \includegraphics[scale=0.58]{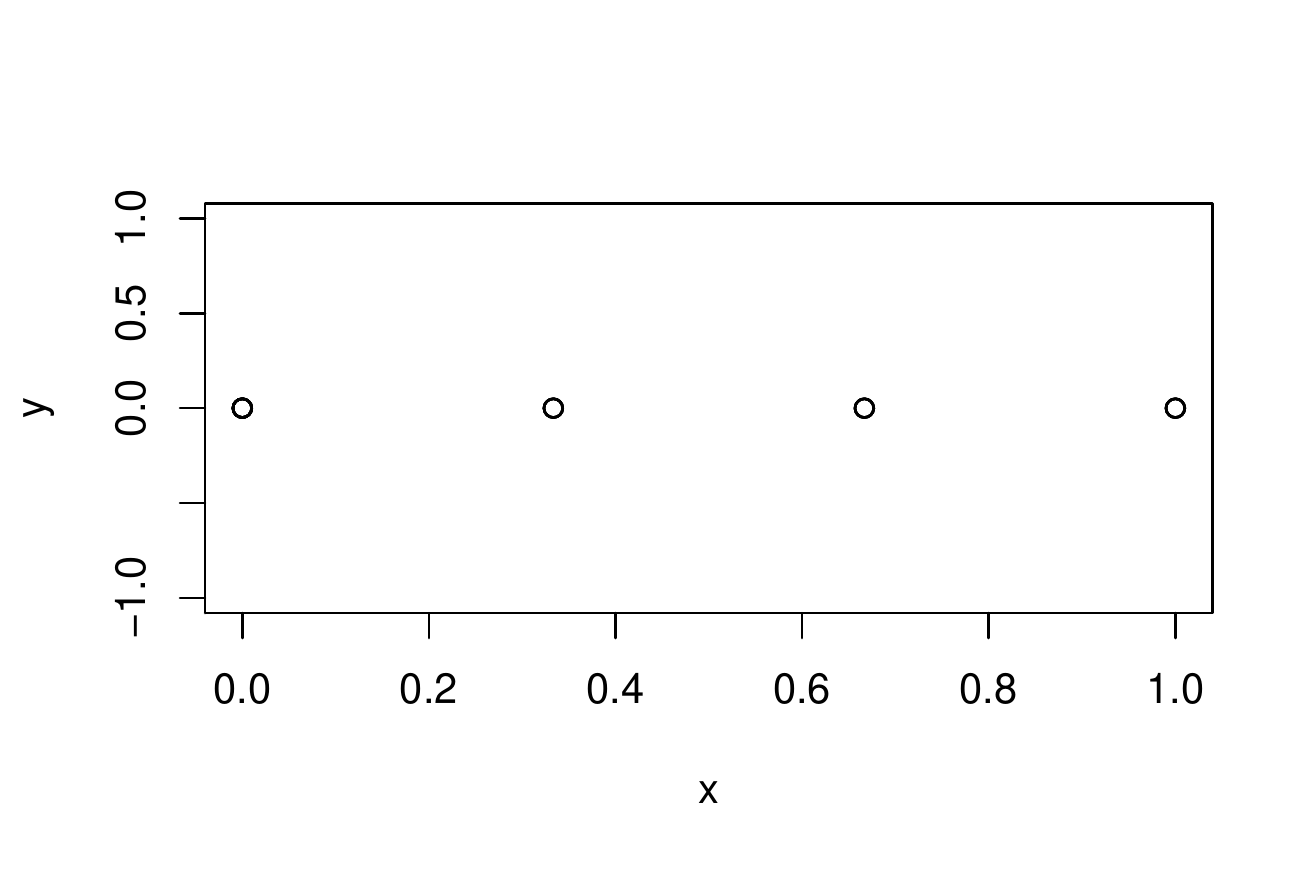}&
\includegraphics[scale=0.58]{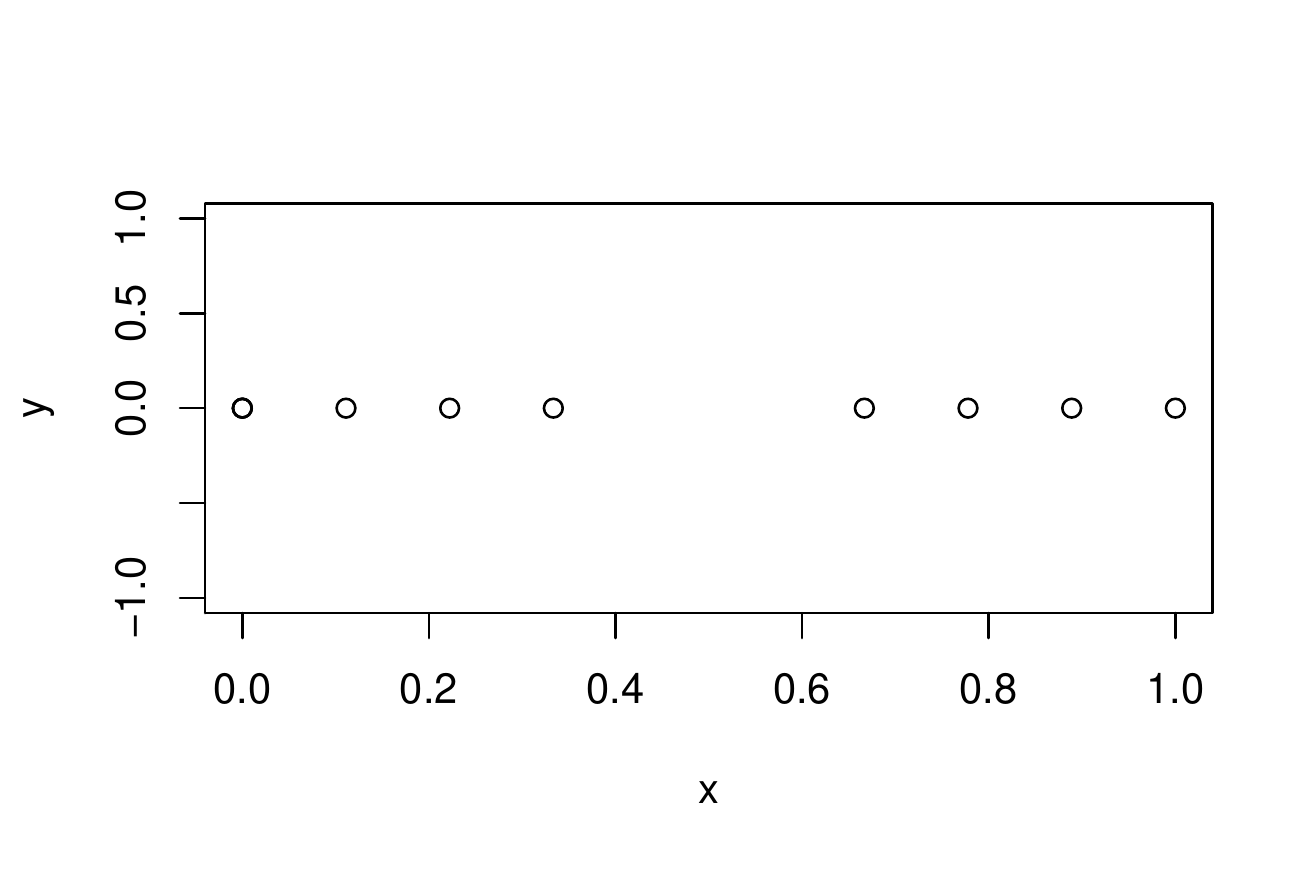}\\
\end{tabular}
\caption{$A_1$ and $A_2$}\label{fig:observationsCantor}
\end{figure}

In Table \ref{table:dimensionesHomologiaCantor} we present the dimension of the homology groups for the first values of the sequence $\{\mathcal{U}_{4\epsilon_n}(A_n) \}_{n\in \mathbb{N}}$. It can be observed that if $n$ increases, then the dimension of the $0$-dimensional homology group for $\mathcal{U}_{4\epsilon_n}(A_n)$ increases. The previous behavior is expected since the Cantor set is totally disconnected. Therefore, in each step we are locating more components. 

\begin{table}[h!]\caption{Dimension of homology groups.}\label{table:dimensionesHomologiaCantor}
\begin{tabular}{c|cclllllll}
{\ul } & $\mathcal{U}_{4\epsilon_{1}}(A_1)$ & $\mathcal{U}_{4\epsilon_{2}}(A_2)$ & $\mathcal{U}_{4\epsilon_{3}}(A_3)$ & $\mathcal{U}_{4\epsilon_{4}}(A_4)$ &   $\mathcal{U}_{4\epsilon_{5}}(A_5)$   &    $\mathcal{U}_{4\epsilon_{6}}(A_6)$   &    $\mathcal{U}_{4\epsilon_{7}}(A_7)$  &   $\mathcal{U}_{4\epsilon_{8}}(A_8)$ \\ \cline{1-9}
$H_0$  & $1$                                & $1$                                & $1$                                & $3$                                & $7$ & $31$ & $63$  & $127 $ \\
$H_1$  & $0$                                & $0$                                & $0$                                & $0$                                & $0$ & $0$  & $0$ & $0$ \\
$H_2$  & $0$                                & $0$                                & $0$                                & $0$                                & $0$ & $0$  & $0$ & $0$
\end{tabular}
\end{table}


\end{ex}

The previous computations, Example \ref{ex:ejemploCuadrados} and Example \ref{ex:ejemploCantor}, have been carried out by the statistical software R. 

\section*{Appendix: Pro-Categories}\label{sec:Appendix}
Given a partially ordered set $(\Lambda,\leq)$, a subset $\Lambda'\subset \Lambda$ is cofinal in $\Lambda$ if for each $\lambda\in \Lambda$ there exists $\lambda'\in \Lambda'$ such that $\lambda<\lambda'$.  A partially ordered set $(\Lambda,\leq)$ is a directed set if for any $\lambda_1,\lambda_2\in \Lambda$, there exists $\lambda\in \Lambda$ such that $\lambda_1\leq \lambda$ and $\lambda_2\leq \lambda$.
\begin{df} Let $\mathcal{C}$ be an arbitrary category. An inverse system in the category $\mathcal{C}$ consists of a directed set $\Lambda$, that is called the index set, of an object $X_\lambda$ from $\mathcal{C}$ for each $\lambda\in \Lambda$ and of a morphism $p_{\lambda, \lambda'}:X_{\lambda'}\rightarrow X_{\lambda}$ from $\mathcal{C}$ for each pair $\lambda\leq \lambda'$. Moreover, one requires that $p_{\lambda,\lambda}$ is the identity map and that $\lambda\leq \lambda'$ and $\lambda'\leq \lambda''$ implies $p_{\lambda,\lambda'}\circ p_{\lambda',\lambda''}=p_{\lambda,\lambda''}$. An inverse system is denoted by $\mathbb{X}=(X_\lambda,p_{\lambda,\lambda'},\Lambda)$, where the $X_\lambda$'s are called the terms and $p_{\lambda,\lambda'}$ are called the bonding maps of $\mathbb{X}$. 
\end{df}
\begin{rem} An inverse system indexed by the natural numbers with its usual order is called an inverse sequence and it is denoted by $\mathbb{X}=(X_n,p_{n,n+1})$. In an inverse sequence $(X_n,p_{n,n+1})$, it suffices to know the morphisms $p_{n,n+1}:X_{n+1}\rightarrow X_{n}$ for every $n\in \mathbb{N}$ because the remaining bonding maps are obtained by composition.
\[\begin{tikzcd}
X_1 & X_2 \arrow{l}{p_{1,2}} & X_3 \arrow{l}{p_{2,3}} & X_4 \arrow{l}{p_{3,4}} & \cdots \arrow{l}{p_4,5} & X_n \arrow{l}{p_{n-1,n}} & \cdots \arrow{l}{p_{n,n+1}} &
\end{tikzcd}
\]
\end{rem}
\begin{rem} Every object $X$ of $\mathcal{C}$ can be seen as an inverse system $(X)$, where every term is $X$ and every  bonding map is the identity map. This inverse system is called rudimentary system.
\end{rem}

A morphism of inverse systems $\mathbb{X}=(X_\lambda,p_{\lambda,\lambda'},\Lambda)\rightarrow \mathbb{Y}=(Y_\mu, q_{\mu,\mu'}, M)$ consists of a function $\phi: M\rightarrow \Lambda$
 and of morphisms $f_\mu:X_{\phi(\mu)}\rightarrow Y_\mu$ in $\mathcal{C}$ for each $\mu \in M $ such that whenever $\mu\leq \mu'$, then there exists $\lambda\in \Lambda$ satisfying that $\lambda\geq \phi(\mu), \phi(\mu')$, for which $f_\mu \circ p_{\phi(\mu),\lambda}=q_{\mu, \mu'}\circ f_{\mu'}\circ p_{\phi(\mu'),\lambda}$, that is, the following diagram commutes.
 
 \[\begin{tikzcd}[row sep=large,column sep=huge]
X_{\phi(\mu)} \arrow{d}{f_{\mu}} &  X_\lambda \arrow[l,"p_{\phi(\mu),\lambda}"']  \arrow{r}{p_{\phi(\mu'),\lambda}} & X_{\phi(\mu')} \arrow{d}{f_{\mu'}} \\
Y_{\mu} & & Y_{\mu'} \arrow{ll}{q_{\mu,\mu'}}
\end{tikzcd}
\]
A morphism of inverse systems is denoted by $(f_\mu,\phi):\mathbb{X}\rightarrow \mathbb{Y}$. Let $\mathbb{Z}=(Z_\nu, r_{\nu,\nu'},N)$ be an inverse system and let $(g_\nu, \psi):\mathbb{Y}\rightarrow \mathbb{Z}$ be a morphism of systems. Then the composition $(g_\nu,\psi)\circ(f_{\mu},\phi)=(h_\nu, \chi):\mathbb{X}\rightarrow \mathbb{Z}$ is given as follows: $\chi=\phi \circ \psi :N\rightarrow \Lambda$ and $h_\nu=g_\nu\circ f_{\psi(\nu)}:X_{\chi(\nu)}\rightarrow Z_{\nu}$. It is routine to check that the composition is well-defined and associative. The identity morphism $(id_{\lambda},id_{\Lambda}):\mathbb{X}\rightarrow \mathbb{X}$ is given by the identity map $id_{\Lambda}:\Lambda\rightarrow \Lambda$ and the identity morphisms $id_\lambda:X_{\lambda}\rightarrow X_{\lambda}$. It is easy to check that $(f_\mu,\phi)\circ (id_{\lambda},id_{\Lambda})=(f_\mu,\phi)$ and $(id_{\mu},id_{M})\circ (f_\mu,\phi)=(f_\mu,\phi)$. Thus, we have a category inv-$\mathcal{C}$, whose objects are all inverse systems in $\mathcal{C}$ and whose morphisms are the morphisms of inverse systems described above.

Let $\mathbb{X}=(X_\lambda,p_{\lambda,\lambda'},\Lambda)$ and $\mathbb{Y}=(Y_\lambda,q_{\lambda,\lambda'}, \Lambda)$ be two inverse systems over the same directed set $\Lambda$. A morphism of systems $(f_\lambda,\varphi)$ is a level morphisms of systems provided $\varphi$ is the identity map and for $\lambda\leq \lambda'$ the following diagram commutes.

 \[\begin{tikzcd}[row sep=large,column sep=huge]
X_\lambda \arrow{d}{f_{\lambda}} & X_{\lambda'}  \arrow[l,"p_{\lambda,\lambda'}"'] \arrow{d}{f_{\lambda'}}\\
Y_\lambda & Y_{\lambda'} \arrow[l,"q_{\lambda,\lambda'}"]
\end{tikzcd}
\]

Given two morphisms $(f_\mu,\phi),(f_\mu',\phi'):\mathbb{X}\rightarrow  \mathbb{Y}$, we write that $(f_\mu,\phi)\sim (f_\mu',\phi')$ if and only if each $\mu\in M$ admits $\lambda\in \Lambda$ satisfying that $\lambda\geq \phi(\mu),\phi'(\mu)$ and that the following diagram commutes.

\[\begin{tikzcd}[row sep=large,column sep=huge]
X_{\phi(\mu)} \arrow[dr,"{f_\mu}"'] & X_\lambda \arrow[l,"p_{\phi(\mu),\lambda}"'] \arrow[r,"p_{\phi'(\mu),\lambda}"] & X_{\phi'(\mu)} \arrow{dl}{f\mu'} \\
& Y_\mu &
\end{tikzcd}
\]

Again, it is routine to check that $\sim$ is an equivalence relation. We have the category pro-$\mathcal{C}$ for the category $\mathcal{C}$. The objects of pro-$\mathcal{C}$ are all inverse systems in $\mathcal{C}$. A morphism $f:\mathbb{X}\rightarrow \mathbb{Y}$ is an equivalence class of morphisms of systems with respect to the equivalence relation $\sim$.

Let $\Lambda$ be a directed set. If $\Lambda'\subset \Lambda$ is a directed set and $\mathbb{X}=(X_\lambda,p_{\lambda,\lambda'},\Lambda)$ is an inverse system, then $\mathbb{X}'=(X_{\lambda},p_{\lambda,\lambda'},\Lambda')$ is a subsystem of $\mathbb{X}$. There is a natural morphism $(i_{\lambda},i)$ given by $i(\lambda)=\lambda$ and $i_\lambda=id_{\lambda}:X_\lambda\rightarrow X_\lambda$ for every $\lambda\in \Lambda'$. The morphism $i:\mathbb{X}'\rightarrow \mathbb{X}$ represented by $(i_{\lambda},i)$ is called the restriction morphism.

\begin{thm} If $\Lambda'$ is cofinal in $\Lambda$, then the restriction morphism $i:\mathbb{X}'\rightarrow \mathbb{X}$ is an isomorphism in pro-$\mathcal{C}$.
\end{thm}
Finally, we recall Morita's lemma, see \cite{morita1974hurewicz} or \cite[Chapter 2, Theorem 5]{mardevsic1982shape}. 
\begin{thm}\label{thm:Moritalema} Let $\mathcal{C}$ be a category and let $\mathbb{X}=(X_\lambda,p_{\lambda,\lambda'},\Lambda)$ and $\mathbb{Y}=(Y_\lambda,q_{\lambda,\lambda'},\Lambda)$ be inverse systems over the same index set $\Lambda$. Let $f:\mathbb{X}\rightarrow \mathbb{Y} $ be a morphism of pro-$\mathcal{C}$ given by a level morphism of systems $(f_\lambda):\mathbb{X}\rightarrow \mathbb{Y}$. Then the morphism $f$ is an isomorphism of pro-$\mathcal{C}$ if and only if every $\lambda\in \Lambda $ admits $\lambda'\geq \lambda$ and a morphism $g_\lambda :Y_\lambda\rightarrow X_\lambda$ of $\mathcal{C}$ such that the following diagram commutes.
\[\begin{tikzcd}[row sep=large,column sep=huge]
X_\lambda \arrow{d}{f_\lambda} & X_{\lambda'} \arrow[l,"p_{\lambda,\lambda'}"'] \arrow{d}{f_{\lambda'}} \\
Y_{\lambda} & Y_{\lambda'} \arrow[l,"q_{\lambda,\lambda'}"] \arrow{ul}{g_\lambda}
\end{tikzcd}
\]
\end{thm}

\bibliography{bibliografia}
\bibliographystyle{plain}

\newcommand{\Addresses}{{
  \bigskip
  \footnotesize

  \textsc{ P.J. Chocano, Departamento de \'Algebra, Geometr\'ia y Topolog\'ia, Universidad Complutense de Madrid, Plaza de Ciencias 3, 28040 Madrid, Spain}\par\nopagebreak
  \textit{E-mail address}:\texttt{pedrocho@ucm.es}

  \medskip

\textsc{ M. A. Mor\'on,  Departamento de \'Algebra, Geometr\'ia y Topolog\'ia, Universidad Complutense de Madrid and Instituto de
Matematica Interdisciplinar, Plaza de Ciencias 3, 28040 Madrid, Spain}\par\nopagebreak
  \textit{E-mail address}: \texttt{ma\_moron@mat.ucm.es}

  \medskip

\textsc{ F. R. Ruiz del Portal,  Departamento de \'Algebra, Geometr\'ia y Topolog\'ia, Universidad Complutense de Madrid and Instituto de
Matematica Interdisciplinar
, Plaza de Ciencias 3, 28040 Madrid, Spain}\par\nopagebreak
  \textit{E-mail address}: \texttt{R\_Portal@mat.ucm.es}

}}

\Addresses

\end{document}